%% file: CLSpdg_arXiv_v1.tex
\documentclass[11pt]{amsart}
\usepackage{color}
 \usepackage{amscd}

\usepackage{amssymb, amsmath}
\usepackage[final]{hyperref}

 \input xy
 \xyoption{all}
 
 \input diagxy

\usepackage{epsf}
\usepackage{epsfig}

\newtheorem{prop}{Proposition}[section]

\theoremstyle{remark} 
\numberwithin{equation}{section}
\newcommand{\Beq}{\begin{equation}}
\newcommand{\Eeq}{\end{equation}}
\newcommand{\Beqr}{\begin{eqnarray}}
\newcommand{\Eeqr}{\end{eqnarray}}
\newcommand{\mpM}{{\mathcal P} {M}}

\newcommand{\mbr}{{\mathbb R}}

\newcommand{\mct}{{\mathcal T }}

\newcommand{\mpm}{{\mathcal P}M}

\newcommand{\oab}{\omega^{(A,B)}}

\newcommand{\om}{\omega}

 \newcommand{\ovA}{{\bar A}}

\newcommand{\ovg}{  {\overline\gamma}}  

\newcommand{\ovG}{  {\overline\Gamma}}

\newcommand{\ovv}{{\overline  v}}

\newcommand{\ovw}{{\overline  w}}

\newcommand{{\wtlg}}{\widetilde\gamma }

\newcommand{{\wtlG}}{\widetilde\Gamma }

\newcommand{{\wtlv}}{\widetilde v }

\newcommand{{\tlg}}{\tilde\gamma }

\newcommand{{\tlG}}{\tilde\Gamma }

 \newcommand{\tlv}{{\tilde v}}

\newcommand{\catdeca}{{\mathbf P}^{\rm dec}_{\bar A}}

\newcommand{\pap}{{\mathcal P}_{\bar A}P}

\newcommand{\papp}{{\mathcal P}_{{\bar A}'}P}

\newcommand{\pdbap}{{\mathcal P}^{\rm dec}_{\bar A}P}

\newcommand{\mbba}{\mathbf {A }}
\newcommand{\mbbb}{\mathbf {B} }

\newcommand{\mbg}{\mathbf {G} }

\newcommand{\mbp}{\mathbf {P} }

\newcommand{\Obj}{{\rm Obj }}
\newcommand{\Mor}{{\rm Mor }}

\begin{document}
\title[Connections on  decorated path space  bundles]{Connections on  decorated path space  bundles}
\author{Saikat Chatterjee }
\address{Saikat Chatterjee, Institut des Hautes \'Etudes Scientifiques \\
35 Route de Chartres\\
Bures-sur-Yvette-91440
France}
\email{saikat.chat01@gmail.com}

\author{Amitabha Lahiri}
\address{Amitabha Lahiri, S.~N.~Bose National Centre for Basic Sciences \\ Block JD,
  Sector III, Salt Lake, Kolkata 700098 \\
  West Bengal, India}
  \email{amitabhalahiri@gmail.com}

\author{Ambar N. Sengupta }
\address{Ambar N. Sengupta, Department of Mathematics\\
  Louisiana State University\\  Baton
Rouge, Louisiana 70803, USA}
\email{ambarnsg@gmail.com}



\begin{abstract}  For a  principal bundle $P\to M$ equipped with a connection $\ovA$, we study an infinite dimensional bundle $\pdbap$ over the space of paths on $M$,  with the points of $\pdbap$ being horizontal paths on $P$ decorated with  elements of a second structure group.  We construct parallel transport processes on such bundles and study   holonomy bundles in this setting. We explain the relationship with categorical geometry and   explore the notion of categorical connections on categorical principal bundles in a concrete  differential geometric way.  \end{abstract}         

\maketitle

\section{Introduction}\label{s:int}

In this paper we explore  differential geometric aspects of categorical principal bundles. Such bundles were introduced in our previous work \cite{CLS2geom} from a category theoretic point of view. The present work, which can be read independently of \cite{CLS2geom}, is a contribution to  the  active literature at the juncture of category theory and differential geometry motivated by ideas in string theory and gauge theories.

The focus of our study is parallel transport on bundles whose elements are paths decorated with elements of a second structure group. We begin with a connection form $\ovA$ on a principal $G$-bundle $\pi:P\to M$, where $G$ is a Lie group, and consider first the structure
\begin{equation}\label{E:piovAbun}
\pi_{\ovA}:\pap\to\mpm:\ovg\mapsto \pi\circ\ovg,
\end{equation}
where $\mpm$ is the space of smooth paths  on $M$ and $\pap$ the space of $\ovA$-horizontal smooth paths on $P$; thus a point $\ovg\in\pap$ is a $C^\infty$ path $[t_0,t_1]\to P$ for which $\ovA_{\ovg(t)}\bigl(\ovg'(t)\bigr)=0$ for all $t\in [t_0,t_1]$.  Figure \ref{fig:Horizpath} illustrates this structure.

%%%%%%%%%%%%%%%%%%%%%%FIGURE HERE%%%%%%%%%%%%%%%%%%%%%%%%%%%%%%%%%%%%%%%%
\begin{figure}[htbp]
\begin{center}
\resizebox{6cm}{!}{\input{gammaba2r.pstex_t}}
\caption{Horizontal paths}
\label{fig:Horizpath}
\end{center}
\end{figure}
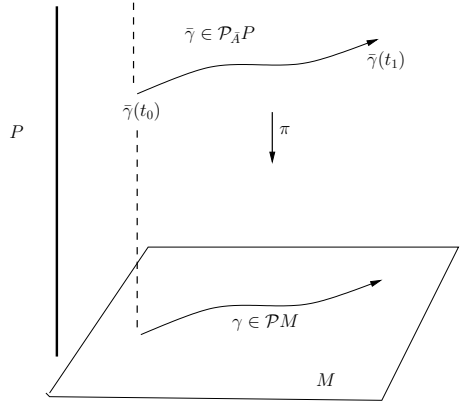
%%%%%%%%%%%%%%%%%%%%%%%%%%%%%%%%%%%%%%%%%%%%%%%%%%%%%%%%%%%%%%%%%%%%%%%%%%

The group $G$ acts on the space $\pap$   by right translations $\ovg\mapsto {\ovg}g$, and   the structure (\ref{E:piovAbun}) has the essential features of a principal $G$-bundle.  Next we introduce a   Lie group $H$ and a semidirect product $H\rtimes_{\alpha}G$, which serves as a `higher' structure group. Using these we construct a {\em decorated  bundle}
\begin{equation}\label{E:pdbapbun}
\pi_{\ovA}^d:\pdbap=\pap\times H\to\mpm:(\ovg, h) \mapsto \pi\circ\ovg,
\end{equation}
where we view each pair $(\ovg, h)$ as an $\ovA$-horizontal path $\ovg$ on $P$ {\em decorated} with an element $h$ drawn from the second structure group $H$. 
It is this structure, illustrated in Figure \ref{fig:decopath}, that is the focus of our work in this paper.  
%%%%%%%%%%%%%%%%%%%%%%%%%FIGURE HERE%%%%%%%%%%%%%%%%%%%%%%%%%%%%%%%%%%%%%%
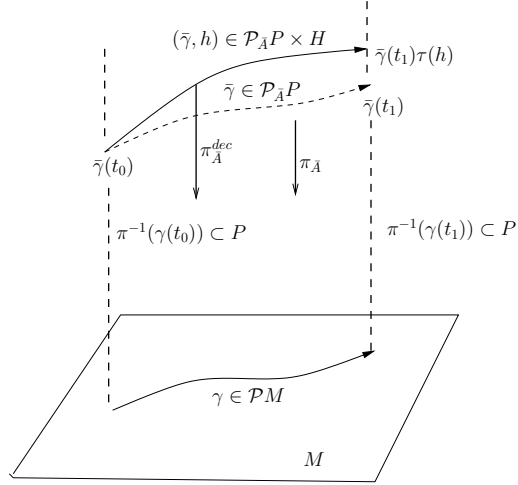
\begin{figure}[htbp]
\begin{center}
\resizebox{6cm}{!}{\input{pardeco.pstex_t}}
\caption{Decorated paths}
\label{fig:decopath}
\end{center}
\end{figure}
%%%%%%%%%%%%%%%%%%%%%%%%%%%%%%%%%%%%%%%%%%%%%%%%%%%%%%%%%%%%%%%%%%%%%%%%%%

We prove results and explain how  this  structure can be viewed as a principal $H\times_{\alpha}G$-bundle.  Parallel transport on this bundle takes a path on $\mpm$ of the form $[s_0,s_1]\to\mpm:s\mapsto \Gamma_s$, represented by a mapping
$$[t_0,t_1]\times [s_0,s_1]\to M: (t,s)\mapsto \Gamma_s(t),$$
and associates to it a path on the decorated bundle  $\pdbap$ of the form
$$[s_0,s_1]\to \pdbap: s\mapsto ({\hat\Gamma}_s, h_s),$$
with specified initial value $({\hat\Gamma}_{s_0}, h_{s_0})$ projecting down to the initial path $\Gamma_{s_0}$ on the base manifold. This parallel transport process is obtained by using certain $1$- and $2$-forms  on $P$ with values in the Lie algebras $L(H)$ and $L(G)$. Given a suitable $1$-form on $P$ with values in the Lie algebra $L(H)$, we can associate, by a type of parallel transport process, a special element $h^*(\ovg)\in H$ for each path $\ovg\in \pap$; this selects out an element $\bigl(\ovg, h^*(\ovg)^{-1}\bigr)\in\pdbap$ for each $\ovg\in\pap$. We then determine, in section   \ref{s:reduct},  conditions on the $1$- and $2$-forms that ensure that parallel transport of a point of $\pdbap$ of the form $\bigl(\ovg, h^*(\ovg)^{-1}\bigr)$  produces an element of the same type.  This investigation is a study of the holonomy bundle for the decorated bundle (the holonomy bundle for a   connection on a traditional finite dimensional principal bundle is a central object in the foundational work of Ambrose and Singer \cite{AmSi1953}).

The background motivation for our work arises from trying to construct a gauge theory for strings joining point particles.  
There is an  active  literature in this area, much of it focused on category theoretic aspects.  In our recent works \cite{CLSpath,  CLS2geom} we have developed a category theoretic framework with focus on  differential geometric notions such as parallel transport over spaces of paths.  In the  present paper we establish a more concrete development of the theory of connections over spaces of paths, explaining the relationship with the category theoretic framework as well.

We have chosen a style of presentation this work intended to save  the reader  the task of  looking back frequently  to earlier and earlier sections for necessary notation and structure needed to understand the statement of a result or a new section; to this end we have often reviewed in summary form the earlier constructions and notation.

The literature in category theoretic geometry is too extensive to offer a  review here.  We mention here the works of  Abbaspour and Wagemann \cite{AbbasWag},  Attal \cite{Attal1, Attal2}, Baez et al.  \cite{BaezSchr, BaezWise}, Barrett \cite{Barr}, Bartels \cite{Bart}, Breen and Messing \cite{BM}, Parzygnat \cite{Parz2014}, Picken et al. \cite{CP, MarPick, MarPick2}, Soncini and Zucchini \cite{SZ}, Schreiber and Waldorf \cite{SW1, SW2} and Viennot \cite{Vien}. These investigations are motivated by the need to develop a `higher gauge theory' and much of the existing literature focuses on category theoretic structures.

\subsection{Results and organization of material}  All our constructions and results in this paper take as background a principal $G$-bundle $\pi:P\to M$, and a given set of connection forms and other forms on $P$.  We denote by $\pap$ the space of all paths on $P$ that are horizontal with respect to a connection $\ovA$ on $P$. The space $\pap$ is itself a principal $G$-bundle over the space $\mpm$ of paths on $M$, in a sense that we shall explain in section \ref{s:pbhor} . We study connections and parallel transport in the bundle $\pap$ and on related bundles. In this context we will  use other symmetry groups $H$ and $K$, along with the initial group $G$.  In a physical context the other symmetry groups would serve as structure groups for gauge theories for interacting strings (described by paths in our framework). 

We begin  in section \ref{s:pbhor} with a description of the bundle of horizontal paths, stating essential ideas, notation and basic results.  We show in Proposition \ref{P:isopullabar} that the bundle of horizontal paths can be understood in terms of a pullback bundle over the space of paths on the base space.  We describe local trivializations of the bundle of horizontal paths. 

In section \ref{s:connhor} we present a connection $\omega$ form on the bundle of horizontal paths. A special case of this type of connection form was studied in our earlier work  \cite{CLSpath}. We work out a differential equation (\ref{E:adotsas}) describing  parallel transport  with respect to this connection.  

We turn next, in section \ref{s:decbun}, to the notion of a {\em decorated bundle}, introduced in a categorical framework in our earlier paper \cite{CLS2geom}. A point on the decroated bundle is of the form
$$({\ovg}, h)$$
where ${\ovg}$ is an $\ovA$-horizontal path on $P$ and $h\in H$ is a `decorating' element attached to ${\ovg}$ (for example, $h$ might arise by integration of an $L(H)$-valued $1$-form along ${\ovg}$).   Here we study the decorated bundle
$$\pdbap\to\mpm$$
 from a differential geometric standpoint and construct local trivializations. The main result here is Proposition \ref{pr:gbundbaradec}, formally stating that $\pdbap$ is a principal $H\rtimes_{\alpha}G$-bundle with specified local trivializations.  The motivation for studying decorated bundles comes from a physics context; the connection $\ovA$ describes a gauge field coupling to point particles and the decorating element allows for coupling of an extended object or string, described by a path $\gamma$ on $M$, to a `higher gauge field' for which the structure group is $H$.

  In section \ref{s:condecbun} we construct a connection form $\Omega$ on the decorated bundle $\pdbap$. 
 In section \ref{s:horliftdec} we determine horizontal lifts with respect to the connection $\Omega$; thus we determine explicitly the equations for parallel transport in the decorated path bundle.
 
 Section \ref{s:catgeom} connects the differential geometric development of this paper to the category theoretic constructions. We describe the fundamental notions of categorical groups and categorical principal bundles. 
 
 Section \ref{s:reduct} is the culmination of all the developments in this paper. Here we make an extensive examination of parallel transport of decorated paths. We consider a special type of decoration of a path $\ovg$, where the decorating element in $H$  arises by means of integration of an $L(H)$ valued  $1$-form along $\ovg$. In Proposition \ref{P:reduction} we find   conditions under which the parallel transport of such a decorated path is itself decorated in the same manner.

\section{A principal bundle of horizontal paths}\label{s:pbhor}

We work with a principal $G$-bundle $\pi:P\to M$, where $G$ is a Lie group, and a connection form $\ovA$ on this bundle. Our focus is on a pair of path spaces, one a space  $\mpm$  of  paths on $M$ and the other a space $\pap$ of $\ovA$-horizontal paths on $P$. The projection map $\pi$ induces a corresponding projection 
$$ \pi_{\bar A}:\pap\to \mpm:{ \overline\gamma}\mapsto \pi\circ { \overline\gamma},$$
while the right action of $G$ on the bundle space $P$ induces a right action of $G$ on $\pap$ that preserves the fibers of $ \pi_{\bar A}$. It is this structure, clearly analogous to  a principal bundle, that we shall study. 
  Specifiying useful topologies and smooth structures on path spaces tend to be   unrewarding tasks, and so we will keep the involvement of such structures to a minimum and make no attempt at formulating or using any general framework for them. However, it is important to note what exactly the elements of $\mpm$ are. Unfortunately, even this requires a somewhat complex articulation of features that are intuitively quite clear. 
  
  The notation and definitions in this section are taken from our earlier works \cite{CLSpath,  CLS2geom}, to which the reader may refer for further details. However, our treatment here is self-contained.

   By a {\em  parametrized path}  on $M$ we mean a $C^\infty$ map $[t_0,t_1]\to M$, for some $t_0,t_1\in\mbr$ with $t_0<t_1$, that is constant near $t_0$ and near $t_1$.  Thus the set of all such paths is
   \begin{equation}\label{E:parmpath}
    \bigcup_{t_0,t_1\in\mbr, t_0<t_1}C_c^\infty\left([t_0,t_1]; M\right),\end{equation}
    where the subscript $c$ signifies the behavior near $t_0$ and $t_1$. If
    $\gamma_1\in C_c^\infty[t_0,t_1]$ and $\gamma_2\in C_c^\infty[t_1,t_2]$ then the composite $\gamma_2\circ\gamma_1$ belongs to $C_c^\infty[t_0,t_2]$.  It is often useful, at least for notational simplicity, to compose paths that are defined on the same parameter domain $[t_0, t_1]$.  With this in mind, we introduce the  quotient set $\mpm$ obtained by identifying paths that differ by a constant time-translation reparametrization. Thus, $\gamma:[t_0,t_1]\to M$ is identified with $\gamma_{+a}:[t_0-a,t_1-a]\to M:t\mapsto \gamma(t+a)$ in $\mpm$; this means that the parametrized paths $\gamma$ and $\gamma_{+a}$ correspond to the same element in $\mpm$.  We will usually not make a notational distinction between $\gamma$ and its equivalence class $[\gamma]$ of such time-translation reparametrizations. We will then use the term `path'  to refer to an equivalence class such as this. 
Many important constructions are invariant under a far larger class of reparametrizations but at this stage we find it more convenient to keep the reparametrizations to a minimum.

Following the notational convention for $\mpm$ we denoted by $\pap$ the set of all $\ovA$-horizontal parametrized paths on $P$, where $\ovA$ is our given connection form. Thus, an element ${ \overline\gamma}\in\pap$ is a $C^\infty$ mapping $[t_0,t_1]\to P$, for some $t_0<t_1$ in $\mbr$, such that
$$\ovA\bigl({ \overline\gamma}'(t)\bigr)=0\qquad\hbox{for all $t\in [t_0,t_1]$.}$$

We shall see now how a local trivialization of the bundle $\pi:P\to M$ gives rise to a local trivialization of the path bundle $\pi:\pap\to \mpm$.   To this end consider an open set $U\subset M$ and a smooth diffeomorphism
\Beq\label{homeo1}
\phi: U\times G\rightarrow \pi^{-1}(U)
\Eeq
that is $G$-equivariant in the sense that $\phi(u,gg')=\phi(u,g)g'$ for all $u\in $ and $g,g'\in G$.  
Associated to $U$ is the set $U^0$ of all paths that begin in $U$: $$U^0={\rm ev}_0^{-1}(U)$$
in the base path space, where ${\rm ev}_0$ gives the initial point of a path:
$$\gamma_0\stackrel{\rm def}{=} {\rm ev}_0(\gamma)\stackrel{\rm def}{=}\gamma(t_0).$$
  We view $U^0$ as an open subset of $\mpm$.  We can construct  a   diffeomorphism between $U^0$ and $U\times G$ by using the trivialization $\phi$; to understand this let 
  \begin{equation}\label{E:gammaAbarp}
  { \overline\gamma}^{\ovA}_{p }
  \end{equation}
  be the $\ovA$-horizontal path on $P$ that starts at $p$ and projects down to $\gamma$; thus
  $$\ovA\Bigl(\{{\overline\gamma}^{\ovA}_p\}'(t)\Bigr)=0\qquad\hbox{for all $t\in [t_0,t_1]$,}$$
  the projection down to the base manifold is
  $$\pi\circ {\overline\gamma}^{\ovA}_p=\gamma,$$
 and the initial point is $p$:
 $${\overline\gamma}^{\ovA}_p(t_0)=p.$$
  Then we define the map:
\begin{equation}\label{E:defphi0}
\phi^0:   U^0\times G\to \pi_{\bar A}^{-1}(U^0):(\gamma,g)\mapsto \phi^0(\gamma,g)\stackrel{\rm def}{=}{ \overline\gamma}^{\ovA}_{p }.
\end{equation}
 
 The mapping $\phi^0$ is $G$-equivariant and is clearly bijective as well.
 
Let us now determine the transition function between trivializations $\phi^0$ and $\psi^0$.  For trivializations $\phi:U\times G\to \pi^{-1}(U)$ and $\psi:V\times G\to \pi^{-1}(V)$ we have the transition function
 $$\theta_{\phi,\psi}:U\cap V\to G$$
 given by
 $$\psi(u,g)=\phi(u,g)\theta_{\phi,\psi}(u)\qquad\hbox{for all $(u,g)\in (U\cap V)\times G$.}$$
Then 
 \begin{equation}
 \begin{split}
\hbox{initial point of $ \psi^0(\gamma,g)$} &=\psi\bigl(\gamma_0, g\bigr)\\
 &=\phi\bigl(\gamma_0, g\bigr)\theta_{\phi,\psi}\bigl(\gamma_0\bigr)\\
 &=\hbox{initial point of $\phi^0(\gamma, g)\theta_{\phi,\psi}\bigl(\gamma_0\bigr)$.}
 \end{split}
 \end{equation}
 Thus the transition function between $\phi^0$ and $\psi^0$ is given by
\begin{equation}
\theta_{\phi^0,\psi^0}:U^0\cap V^0\to G:\gamma\mapsto \theta_{\phi,\psi} \bigl({\rm ev}_0(\gamma)\bigr).
\end{equation}
(Technically, we have not imposed a topology on the path space and so we do not have to verify continuity or smoothness of these transition functions.)
 
 In fact we can describe the bundle $\pi:\pap\to \mpm$ quite simply as a pullback of the bundle $\pi:P\to M$ under the evaluation map
 $${\rm ev}_0:\mpm\to M: \gamma\mapsto \gamma_0={\rm ev}_0(\gamma).$$
 To this end let
\begin{equation}\label{pullbund1}
{\rm ev}^{*}_0 P:=\{( \gamma, p)\in \mpm \times P\,|\,\gamma_0 =\pi(p)  \}.
\end{equation}
Thus a point of ${\rm ev}^{*}_0 P$ is specified by a point $p\in P$ along with a path $\gamma$ on $M$ that starts at $\pi(p)$.

The group $G$ acts on this space by 
$$( \gamma, p)g:=(\gamma, pg).$$
The mapping
$$\pi_0:{\rm ev}^{*}_0 P\to\mpm: ( \gamma, p)\mapsto\gamma $$
is a surjective projection for which 
$$\pi_0\bigl((\gamma)g, p\bigr)=\pi_0( \gamma, p)
$$
for all $p\in P$, $\gamma\in\mpm$ and $g\in G$.

\begin{prop}\label{P:isopullabar}
The mapping
\begin{equation}\label{E:defmu}
\mu: {\rm ev}^{*}_0 P\rightarrow {\mathcal P}_{\bar A}P:(\gamma,p)\mapsto{ \overline\gamma}{\ovA}_p,
\end{equation}
where ${ \overline\gamma}^{\ovA}_p$ is the $\ovA$-horizontal lift of $\gamma$ initiating at $p\in P$, is
   a   $G$-equivariant bijection $\mu:{\rm ev}^{*}_0 P\rightarrow {\mathcal P}_{\bar A}P$. The diagram
\begin{equation} \label{dia:comiso}
\xymatrix{
         \ar[d]^{\pi_0} {{\rm ev}^{*}_0 P}      \ar[r]^-{\mu} & \pap\ar[d]_{\pi_{{\ovA}}} \\
\mpm\ar[r]_-{{\rm id}}& \mpm
}
\end{equation}

commutes.
\end{prop}
Let us note that $\pi_{\ovA}:\pap\to\mpm$ is a principal $G$-bundle (in the sense that the projection $\pi$ is a surjection and  the group $G$ acts freely and transitively on each fiber of $\pi$); since the base and bundle spaces are both path spaces it might seem at first that the structure group is infinite dimensional but in fact it is just $G$ because we are focusing on the $\ovA$-horizontal paths. 
\begin{proof}
Consider any $(\gamma, p)\in {\rm ev}^{*}_0 P$; then by definition $p\in \pi^{-1}(\gamma(0))$. Hence we 
can horizontally lift the path $\gamma$ on $M$ to $P$ by $\bar A$ to obtain the $\bar A$-horizontal 
path ${\ovg}^{\bar A}_p$ starting from $p\in P$ and that path is unique. On the other hand, any element $\ovg\in {\mathcal P}_{\bar A}P$ is the image under $\mu$ of  $\bigl(\gamma,  { \overline\gamma}_0 \bigr)\in  {\rm ev}^{*}_0 P$, where $\gamma= \pi_{\bar A}({\ovg})\in\mpm$.
Thus   $$\mu:( \gamma, p)\mapsto {\ovg}_p $$ 
is a bijection. Since horizontal lifts behave equivariantly under the action of $G$, the mapping $\mu$ is $G$-equivariant.  The definition of $\mu$ also implies directly that the diagram \eqref{dia:comiso} is commutative.
\end{proof}

We define a {\em tangent vector} $\ovv$ at ${\ovg}\in {\mathcal P}_{\bar A}P$ with the following description:
\begin{itemize}
\item[(i)] $\ovv$ is a $C^\infty$ vector field along $\ovg$,
\item[(ii)] ${\ovv}$ satisfies {\em tangency condition}
\Beq\label{E:tang}
\frac{\partial {\bar A}({\ovv}(t))}{\partial t}=F^{\bar A}({\ovg}'(t),{\ovv}(t)),
\Eeq
for all $t\in [t_0,t_1]$, where $[t_0,t_1]$ is the domain of ${ \ovg}$, and
\item[(iii)] $\ovv$ is constant near $t_0$ and near $t_1$.
\end{itemize}
 To be more precise two such vector fields are viewed as the same tangent vector if they differ by reparametrization by a translation as in the discussion following (\ref{E:parmpath}). We denote the set of all such vector fields by
$$T_{ \overline\gamma}\pap.$$
The linear nature of the differential equation (\ref{E:tang}) implies that this tangent space is indeed a vector space (closed under addition and scaling).

 There is a unique tangent vector   $\ovv\in T_{\ovg}\pap$ with a specified projection vector field $v=\pi_*{ \overline v}$ and   initial value ${ \overline v}(t_0)$. We review the proof from  \cite[Lemma 2.1] {CLSpath}. Let  $ {\tilde v}(t)^h$ be the $\ovA$-horizontal vector  in $  T_{\ovg(t)}P$ that  projects by $\pi_*$ to $v(t)$:
 \begin{equation}
 \ovA\bigl({\tilde v}(t)^h\bigr)=0\quad\hbox{and}\quad \pi_*{\tilde v}(t)^h=v(t).
 \end{equation}
 Now let
 $$Z(t)=  \ovA\bigl(\ovv(t_0)\bigr)+  \int_{t_0}^t F^{\ovA}\bigl(\ovg'(u), {\tilde v}(u)\bigr)\,du\in L(G).$$
   Now consider the vector field $\ovv$ along $\ovg$ specified by
  \begin{equation}\label{E:uniquelift}
\ovv(t)= {\tilde v}(t)^h+\ovg(t)Z(t)
 \end{equation}
 for all $t\in [t_0,t_1]$. Then 
 \begin{equation}\label{E:ovAtZt}
 \ovA\bigl(\ovv(t)\bigr)=Z(t)=  \ovA\bigl(\ovv(t_0)\bigr)+  \int_{t_0}^t F^{\ovA}\bigl(\ovg'(u), {\tilde v}(u)\bigr)\,du.
 \end{equation}
 Thus the differential equation (\ref{E:tang}) holds. Moreover, if $\ovv$ is any vector field along $\ovg$ projecting down to $v$ and   satisfying the differential equation  (\ref{E:tang}) then the relation (\ref{E:ovAtZt}) holds and so $\ovv(t)$ is given by (\ref{E:uniquelift}).  Because $\gamma$ is constant near $t_0$ and $t_1$ we see that $Z$ is constant near these endpoints. Similarly $v$ and ${\tilde v}$ are also constants near the ends, and hence so is $\ovv$.

 Let us see how this is consistent with the pullback point of view in Proposition \ref{P:isopullabar}. If ${ \overline v}_0\in T_pP$, where $p={ \overline\gamma}(t_0)$ and if $v$ is a smooth vector field along $\gamma=\pi_{\ovA}({ \overline\gamma})$, viewed as a vector in $T_{\gamma}\mpm$, with initial value $v(t_0)=d\pi|_{p}{ \overline v}_0$, then 
\begin{equation}\label{E:Dmu}
d\mu\big|_{(\gamma,p)}(v, { \overline v}_0)={ \overline v},
\end{equation}
where the derivative $d\mu$ is taken in a formal but natural sense; more officially, we can take (\ref{E:Dmu}) as defining $d\mu$.

Working with  the map $\mu$ that identifies $\pap$ with the pullback bundle ${\rm ev}_0^*P$ it is possible to construct   local trivializations of $\pi_{\ovA}:\pap\to\mpm$ from those of $\pi:P\to M$, and these coincide with the type given in (\ref{E:defphi0}).

\subsection{Changing the base connection $\pap$} We have been working with a fixed connection $\ovA$ on the principal $G$-bundle $\pi:P\to M$ and using this we have defined the path bundle $\pap\to\mpm$.  Changing $\ovA$ to another connection 
$ \ovA'$
 produces a bundle ${\mathcal P}_{\ovA'}P \to\mpm$. We show now that this is `isomorphic' to $\pap\to\mpm$; this is, of course, quite natural from the point of view of Proposition \ref{P:isopullabar}.  For the following result let us recall  from (\ref{E:tang}) that the tangent space
 $$T_{\ovg}\pap$$
 consists of all vector fields $\ovv$ along $\ovg\in\pap$, constant near the initial and final points,  that satisfy the condition
 \begin{equation}\label{E:tang2}
 \frac{\partial \ovA\bigl(\ovv(t)\bigr)}{\partial t}=F^{\ovA}\bigl(\ovg'(t), \ovv(t)\bigr)\qquad\hbox{for all $t\in [t_0,t_1]$,}
 \end{equation}
 where $F^{\ovA}$ is the curvature form of $\ovA$.

  \begin{prop}\label{P:isomor} Let $\ovA$ and $\ovA'$ be   connections on a principal $G$-bundle $\pi:P\to M$, and $C$ an $L(G)$-valued Ad-equivariant $1$-form on $P$ that vanishes on vertical vectors.  Let  $\pap$ be the set of all paths on $P$ that are $\ovA$-horizontal and $\papp$ the set of all $\ovA'$-horizontal paths.  For each $\ovg\in\pap$ let $\mct(\ovg)$ be the path on $P$ that is $\ovA'$-horizontal, has the same initial point as $\ovg$, and projects down to the same path $\pi\circ\ovg$ on $M$ as $\ovg$. Then the mapping
 \begin{equation}
\label{functF}
\mct:\pap\to\papp:\ovg\mapsto {\mct}(\ovg)
\end{equation} 
is a bijection. For any vector field $\ovv$ along   $\ovg$ that belongs to the tangent space $T_{\ovg}\pap$ let 
\begin{equation}\label{E:Tstarv}
\mct_*\ovv\in T_{\mct(\ovg)}(\papp),
\end{equation}
 be the vector field along $\mct(\ovg)$ whose initial value is $\ovv(t_0)$ and whose projection by $\pi_*$ is the vector field $\pi_*\circ \ovv$ along the path $\gamma=\pi\circ\ovg\in\mpm$.  Then
 \begin{equation}\label{E:Tstarvhor}
 \mct_*(\ovv)(t) - \ovv(t)g_{\ovg}(t)
 \end{equation}
 is a vertical vector, where $g_{\ovg}(t)$ is the element of $G$ for which
 $$\mct(\ovg)(t)=\ovg(t)g_{\ovg}(t).$$
   \end{prop}
   
   It is natural to think of $ \mct_*(\ovv)$ as the image of $v\in T_{\ovg}\pap$ under the `derivative' of $\mct$. In more detail, suppose
   $$\tlG:[t_0,t_1]\times [s_0,s_1]\to P: (t,s)\mapsto\tlG_s(t)$$
   is a $C^\infty$ map for which $\ovg=\pi\circ\tlG_{s_0}$ and 
   $$\ovv(t)=\partial_s{\tlG}_s\big|_{s=s_0}(t)\qquad\hbox{for all $t\in [t_0,t_1]$.}$$
    This displays the vector field $\ovv\in T_{\ovg}\pap$ as the `tangent vector' to a path $s\mapsto\tlG_s$ on $\pap$.  Then the image of $\ovv$ under the derivative of $\mct$ at $\ovg$ should be the tangent, at $s=s_0$, of the image path
    $$s\mapsto \mct(\tlG_s).$$
    This tangent is the vector field along $\mct(\ovg)$ given by
    $$t\mapsto w(t)\stackrel{\rm def}{=} \partial_s \mct(\tlG_s)(t)\Big|_{s=s_0}.$$
    Focussing on the initial `time' $t=t_0$ we have
    $$w(t_0)= \partial_s{\tlG}_s(t_0)\Big|_{s=s_0}$$
    because $\mct(\tlG_s)(t_0)=\tlG_s(t_0)$ by definition of  the mapping $\mct$. Thus
    $$w(t_0)= \ovv(t_0).$$
   Thus $w\in T_{\mct(\ovg)}\papp$, being uniquely  determined by the initial value $w(t_0)$ and the projection $\pi_*\circ w= \pi_* \ovv\in T_{\pi\circ\ovg}M$, is exactly $\mct_*(\ovv)$ as we have defined it above.

 \begin{proof} Horizontal lift  of a path by a connection is uniquely determined by the initial point of the lifted path. Using this we see that $\ovg$ is determined uniquely when $\mct(\ovg)$ is known. Thus $\mct$ is a bijection.
 
 The right action mapping
 $$R_g: P\to P:p\mapsto pg,$$
 for any fixed $g\in G$,  preserves fibers:
 $$\pi\circ R_g(p)=\pi(p)\qquad\hbox{for all $p\in P$.}$$
 Taking the derivative at $p$ on any vector $v\in T_pP$ we then have
 \begin{equation}
 \pi_*|_{pg}\Bigl((R_g)_*|_pv\Bigr)=\pi_*|_pv
 \end{equation}
 for all $v\in T_pP$ and all $p\in P$. Our notation $vg$ means simply $(R_g)_*|_pv$:
 $$vg\stackrel{\rm def}{=}(R_g)_*|_pv.$$
 Hence
 $$\pi_*(vg)=\pi_*(v).$$
 Next from the definition of  $\mct_*(\ovv)(t)$ we know that its projection by $\pi_*$ is the vector $\pi_*(\ovv(t))\in T_{\pi\circ\ovg(t)}M$. Thus
 the vectors $\mct_*(\ovv)(t)$ and $\ovv(t)g_{\ovg}(t)$ in $T_{\mct(\ovg)(t)}P$ both project down by $\pi_*$ to the vector $\pi_*(\ovv(t))\in T_{\pi\circ\ovg(t)}M$. Hence the difference $ \mct_*(\ovv)(t) - \ovv(t)g_{\ovg}(t)$ is a vertical vector.
 \end{proof}
  
\subsection{Pullback of forms} We continue with the comparison of the horizontal path spaces $\pap$ and $\papp$ using the mapping $\mct$. We define pullbacks in the natural way: if ${\tilde D}$ is a $k$-form on $\papp$, with values in some vector space, then $\mct^*{\tilde D}$ is the $k$-form on $\pap$ given by
\begin{equation}\label{E:kformpullback}
(\mct^*{\tilde D})({\tilde v}_1,\ldots, {\tilde v}_k)={\tilde D}(\mct_*{\tilde v}_1,\ldots, \mct_*{\tilde v}_k).
\end{equation}
For example, suppose $B$ is $2$-form on $P$ with values in some vector space. Consider then the $2$-form ${\tilde B}$ on $\papp$ given by
\begin{equation}
{\tilde B}_{{\tilde\gamma}}({\tilde v}, {\tilde w})=\int_{t_0}^{t_1}B_{{\tilde\gamma}(t)}\bigl({\tilde v}(t), {\tilde w}(t)\bigr)\,dt
\end{equation}
for all ${\tilde\gamma}\in\papp$ and ${\tilde v}, {\tilde w}\in T_{{\tilde\gamma}}\papp$. Then 
\begin{equation}\label{E:tildB}
\bigl(\mct^*{\tilde B}\bigr)|_{\ovg}(\ovv, \ovw)=\int_{t_0}^{t_1}B_{\mct(\ovg)}\bigl((\mct_*{\ovv})(t), (\mct_*{\ovw})(t)\bigr)\,dt.
\end{equation}
We can also pullback the $1$-form on $\papp$ given by the Chen integral
\begin{equation}\label{E:Bchen}
B^{\rm ch}_{\tlg}(\tlv)\stackrel{\rm def}{=}\int_{t_0}^{t_1}B\bigl(\tlg'(t), \tlv(t)\bigr)\,dt
\end{equation}
to obtain the $1$-form $\mct^*B^{\rm ch}$ given by
\begin{equation}\label{E:Bchen2}
(\mct^*B^{\rm ch})_{\ovg}(\ovv)\stackrel{\rm def}{=}\int_{t_0}^{t_1}B\bigl(\ovg'(t), \ovv(t)\bigr)\,dt
\end{equation}
As another example, for a $1$-form $C$ on $P$, with values in some vector space, we have a $1$-form ${\tilde C}_0$ on $\papp$ given by
$${\tilde C}_0|_{{\tilde\gamma}}({\tilde v})=C_0|_{{\tilde\gamma}(t_0)}\bigl({\tilde v}(t_0)\bigr)$$
and then the pullback $\mct^*{\tilde C}_0$ is given by
\begin{equation}\label{E:tildC}
\bigl(\mct^*{\tilde C}_0\bigr)|_{\ovg}(\ovv)= {\tilde C}_0|_{\mct({\tilde\gamma})(t_0)}\bigl((\mct_*\ovv)(t_0)\bigr),
\end{equation}
for all $\ovg\in\pap$ and $\ovv\in T_{\ovg}\pap$.

 \section{A connection form on the space of horizontal paths}\label{s:connhor}

 We continue working with a principal $G$-bundle
 $$\pi:P\to M$$
 equipped with a connection form $\ovA$, and the corresponding projection map
 $$\pi_{\ovA}:\pap\to\mpm: { \overline\gamma}\mapsto \pi\circ { \overline\gamma},$$
 where $\mpm$ is the  space of smooth paths on $M$ and $\pap$ the  space of smooth horizontal paths in $P$. Here and always we identify a parametrized path $\gamma:[t_0,t_1]\to X$, in a space $X$, with the reparametrized path
 $[t_0+a,t_1+a]\to X:t\mapsto \gamma(t-a)$, for any $a\in\mbr$, so that they correspond to the same element in the path space over $X$.  More technically, we  work with the  space obtained by quotienting the space of parametrized paths by all the constant translations of the `time' parameter domain.
  
 In the preceding section we have seen how $\pi_{\ovA}:\pap\to\mpm$ can be viewed, in a reasonable sense, as a principal $G$-bundle. Now we turn to a description of a $1$-form $\om$ on $\pap$ (the sense in which this is a $1$-form will become clear) that essentially provides a connection form on this path space bundle.

\subsection{The $2$-form ${B_0}$ and a connection form $A$}
Henceforth we will work with an $L(G)$-valued $2$-form ${B_0}$ on $P$ that has the following two special properties:
\begin{itemize}
\item[(i)] ${B_0}$ is ${\rm Ad}$-equivariant in the sense that
$${B_0}|_{pg}(vg, wg)  ={\rm Ad}(g^{-1}){B_0}(v, w),$$
for all $p\in P$ and all $v, w\in T_pP$;
\item[(ii)]  ${B_0}$  is horizontal in the sense that
$${B_0}(v, w)  =0$$
whenever $v$ or $w$ is a vertical vector (a vector that projects by $\pi_*$ to $0$).
\end{itemize}
Thus ${B_0}$ satisfies 
\Beq\label{alequi}
\begin{split}
 {B_0}|_{pg}(vg, wg) &={\rm Ad}(g^{-1}){B_0}(v, w), \qquad \forall v, w\in T_pP, g\in G,\\
 {B_0}(v, w) &=0, \qquad {\rm if}\hskip 0.1 cm  v\hskip 0.1 cm {\rm or}\hskip 0.1 cm w \hskip 0.1 cm{\rm is \hskip 0.1 cm vertical}
\end{split}
\Eeq
at all points $p\in P$.

As our final ingredient, let $A$ be a connection form on the $G$-bundle $\pi:P\to M$. 

\subsection{The  form $\oab$ on $\pap$}
We define an $L(G)$-valued $1$-form $\omega^{(A, {B_0})}$ on ${\mathcal P}_{\bar A}P$
by
\Beq\label{conn}
\omega^{(A, {B_0})}_{\overline \gamma}({\overline v}):=A\bigl({\overline v}(t_0))+  \int_{t_0}^{t_1}{B_0}({\overline v}(t), {\overline \gamma}'(t)\bigr)\, dt,
\Eeq
where ${\overline v}\in T_{{\overline \gamma}}{\mathcal P}_{\bar A}P$.   

\subsection{The connection form $\om$ on $\pap$}
In $\oab$ we have given a special role to the left endpoint $\ovg(t_0)$; this could, however,  be replaced by the right endpoint $\ovg(t_1)$.  In fact a slightly more general construction leads to a $1$-form with dependence on both endpoints.

Adding on a term depending on the right endpoint value $\ovv(t_1)$ through the $1$-form ${C^R_0}$ we have the $L(G)$-valued $1$-form $\omega=\omega^{(A,{B_0}, {C^L_0}, {C^R_0})}$ on ${\mathcal P}_{\bar A}P$
 given by
 \Beq\label{E:defomabcr}
\omega_{\overline \gamma}({\overline v}):=A_{\ovg(t_0)}\bigl({\overline v}(t_0)\bigr)+{C^R_0}|_{\ovg(t_1)}\bigl({\overline v}(t_1)\bigr)-{C^L_0}|_{\ovg(t_0)}\bigl({\overline v}(t_0)\bigr)+  \int_{t_0}^{t_1}{B_0}|_{\ovg(t)}\bigl({\overline v}(t), {\overline \gamma}'(t)\bigr)\, dt.
\Eeq
Thus
\begin{equation}\label{E:ooabc}
\omega=\oab+{\rm ev}_{1}^*{C^R_0}-{\rm ev}_{0}^*{C^L_0},
\end{equation}
where ${\rm ev}_{0}$ and ${\rm ev}_{1}$ are the evaluations at the left and the right endpoints respectively. The forms ${C^L_0}$ and ${C^R_0}$ on $P$  take values in $L(G)$, vanish on vertical vectors and are equivariant:
\begin{equation}\label{E:CLRequiv}
C^{L, R}_0|_{pg}(vg)={\rm Ad}(g^{-1})C^{L, R}_0|_p(v)
\end{equation}
for all $p\in T_pP$, $v\in T_pP$ and $g\in G$.

The additional terms in (\ref{E:ooabc}) allow  us to include counterparts of $\oab$ that have a right endpoint term in place of the left endpoint term $A\bigl({\overline v}(t_0)\bigr)$ by taking
\begin{equation}
 {C^L_0}={C^R_0}= A-\ovA 
\end{equation}
and replacing $B_0$ by $F^{\ovA}+B_0$ leads to the counterpart of $\oab$ involving the right endpoint in place of the left endpoint.

\begin{prop}\label{pr:connection}
The $1$-form $\omega$ on the principal $G$-bundle $ \pi_{\bar A}: {\mathcal P}_{\bar A}P\to \mpM $ has the following properties:
\begin{itemize}
\item[(i)] 
\begin{equation}\label{E:oABvg}
\omega({ \overline v}g)={\rm Ad}(g^{-1})\omega({ \overline v})
\end{equation}
for all $g\in G$ and all  vector fields ${ \overline v}\in T_{{ \overline\gamma}}\pap$ and all ${ \overline\gamma}\in\pap$;
\item[(ii)] If $Y$ is any element of the Lie algebra $L(G)$ and ${ \overline Y}$ is the vector field along ${ \overline\gamma} $ given by ${ \tilde Y}(t)=\frac{d}{du}\big|_{u=0}{ \overline\gamma}(t)\exp(uY)$, then
\begin{equation}\label{E:oABvert}
\omega({ \tilde Y})=Y.
\end{equation}
\end{itemize}
\end{prop}
The property  (\ref{E:oABvert}) can also be written as:
\begin{equation}\label{E:oABvert2}
\omega_{\ovg}(\ovg Y)=Y \qquad\hbox{for all $Y\in L(G)$ and $\ovg\in\pap$.}
\end{equation}

In  \cite[Proposition 2.2 ]{CLSpath} we have established the preceding result with ${C^R_0}=0$. The proof is simple, so we present a quick sketch here. The equivariance property (i) holds for each of the terms on the right in the definition of $\omega$ in (\ref{E:defomabcr}) and so it holds for $\omega$. Next, applying $\omega$ to the vector ${ \tilde Y}\in T_{\ovg}\pap$ all terms on the right in  (\ref{E:defomabcr}) are $0$ except for the very first one which equals $A\bigl(\ovg(t_0){\tilde Y}\bigr)=Y$ since $A$ is a connection form on $P$; this establishes property (ii).

Properties (i) and (ii) are the essential properties of a connection form on a traditional principal bundle and so we will say that $\omega$ is a {\em connection} on $\pi:\pap\to\mpm$ even though we have not equipped the latter  with a smooth structure.  We will use this terminology henceforth for other forms that  enjoy the properties (i) and (ii) in the relevant bundles.

\subsection{Horizontal lifts of vectors using $\omega$} Consider a path $\gamma\in\mpm$ and a tangent $v\in T_{\gamma}\mpm$; this is just a smooth vector field along $\gamma$. Our objective now is to show that for any $\ovg\in\pi_{\ovA}^{-1}(\gamma)$,  the connection $\om$ provides a unique horizontal lift
 ${ \overline v} \in T_{{\overline \gamma}}\pap$;
that is, ${ \overline v} $ satisfies
\Beq\label{omegvectlift}
\om_{\ovg}\bigl({ \overline v}\bigr)=0 \qquad\hbox{and}\qquad \pi_{\ovA}({ \overline v})=v.
\Eeq

We can choose a $C^\infty$ vector field ${\ovv}$ along the path $\ovg$ for which
$$d\pi_{\ovg(t)}\bigl({\ovv}_{\ovg(t)}\bigr)=v_{\gamma(t)}\qquad\hbox{for all $t\in [t_0,t_1]$.}$$
Now let $Z_0$ be the element of $L(G)$ given by
\begin{equation}\label{E:defZ}
Z_0=-\Bigl[{C^R_0}|_{\ovg(t_1)}\bigl({\ovv}(t_1)\bigr)-{C^L_0}|_{\ovg(t_0)}\bigl({\ovv}(t_0)\bigr)+   \int_{t_0}^{t_1}{B_0}|_{\ovg(t)}\bigl({\ovv}(t), {\overline \gamma}'(t)\bigr)\, dt 
\Bigr].
\end{equation}
Let  $\ovv_0\in T_{\ovg(t_0)}P$ for which
\begin{equation}\label{E:ovv0def}
\ovv_0=   {\ovv}^h_0+\ovg(t_0)Z_0,
\end{equation}
where 
$$\ovv^h_0\in T_{\ovg(t_0)}P$$
is the unique $\ovA$-horizontal vector which projects by $\pi_*$ down to $v(t_0)$. Now let $\ovv$ be the vector field along $\ovg$ that is in the tangent space $T_{\ovg}\pap$ and has initial value $\ovv_0$; this vector field is specified in (\ref{E:uniquelift}) discussed earlier. Thus
\begin{equation}
\ovA_{\ovg(t_0)}\bigl(\ovv(t_0)\bigr)=\ovA_{\ovg(t_0)}(Z_0)
\end{equation}
and so
\begin{equation}\label{E:Av0Z0}
\begin{split}
& \ovA_{\ovg(t_0)}\bigl(\ovv(t_0)\bigr)+{C^R_0}|_{\ovg(t_1)}\bigl({\ovv}(t_1)\bigr)-{C^L_0}|_{\ovg(t_0)}\bigl({\ovv}(t_0)\bigr) 
 +   \int_{t_0}^{t_1}{B_0}|_{\ovg(t)}\bigl({\ovv}(t), {\overline \gamma}'(t)\bigr)\, dt \\
&=0.
\end{split}
\end{equation}
This says precisely that
$$\om_{\ovg}(\ovv)=0.$$
Thus we have shown existence of the $\om$-horizontal lift $\ovv\in T_{\ovg}\pap$ of the vector field $v\in T_{\gamma}\mpm$. Uniqueness follows from the fact that the condition (\ref{E:Av0Z0}) implies that the initial value $\ovv(t_0)$ is given by $\ovv_0$ as in (\ref{E:ovv0def}
), and this uniquely specifies $\ovv$ as discussed in the context of  (\ref{E:uniquelift}).

\subsection{Parallel transport of horizontal paths by $\om$}\label{ss:omhortr}  Let us now understand the process of parallel transport in the bundle $\pap$ using the connection form $\om$; we recall that
\begin{equation}
\om_{\overline\gamma} ({\overline v})=A\bigl({\overline v}(t_0)\bigr)+{C^R_0}\bigl({\overline v}(t_1)\bigr)-{C^L_0}\bigl({\overline v}(t_0)\bigr) +  \int_{t_0}^{t_1}{B_0}\Bigl({\overline v}(t), {\overline\gamma} '(t)\Bigr)\,dt,
\end{equation}
for any ${\overline\gamma}\in \pap$ and ${\overline v}\in T_{\overline\gamma}\pap$.  
Consider a $C^\infty$ map
$$[t_0,t_1]\times[s_0,s_1]\to M: (t,s)\mapsto {\Gamma}(t,s)={\Gamma}_s(t),$$
forming a path $s\mapsto\Gamma_s$  on $\mpm$, and consider an initial 
`point' ${\tilde\Gamma}_{s_0}\in\pap$ with 
$$\pi\circ{\tilde\Gamma}_{s_0}={\Gamma}_{s_0}.$$
 Now let
$$\ovG:[t_0,t_1]\times[s_0,s_1]\to P: (t,s)\mapsto \ovG_s(t)$$
be the mapping specified by the requirements that each path 
$$t\mapsto \ovG_s(t)$$
be $\ovA$-horizontal (hence $\ovG_s\in\pap$) and the initial points $\ovG_s(t_0)$ trace out an $A$-horizontal path on $P$ with initial point ${\tilde\Gamma}_{s_0}(t_0)$:
\begin{equation}\label{E:asspec}
\begin{split}
 s &\mapsto\ovG_s(t_0)\in P \qquad \hbox{is an $A$-horizontal path,} \\
 \ovG_{s_0} &={\tilde\Gamma}_{s_0}.
 \end{split}
 \end{equation}
By the nature of the differential equation for parallel transport the path $s\mapsto {\ovG}_s(t_0)$ is $C^\infty$ and then so is the mapping $\ovG$.

Now we would like to understand the nature of the path
$$s\mapsto {\tilde\Gamma}_s\in\pap$$
that is the $\om$-horizontal lift of $s\mapsto \Gamma_s$.  Since both ${\tilde\Gamma}_s$ and  $\ovG_s$ are $\ovA$-horizontal and both project down to the same path $\Gamma_s\in\mpm$ we can express ${\tilde\Gamma}_s$ as a rigid shift of $\ovG_s$:
\begin{equation}\label{E:defasG}
{\tilde\Gamma}_s=\ovG_sa_s
\end{equation}
for a unique $a_s\in G$, for each $s\in [s_0,s_1]$. The tangent vector to $\pap$ for the derivative of $s\mapsto {\tilde\Gamma}_s$ is the vector field along ${\tilde\Gamma}_s$ given by
\begin{equation}\label{E:tangGas}
t\mapsto \partial_s{\tilde\Gamma}_s(t)= \partial_t{\ovG}_s(t)a_s+   {\ovG}_s(t)a_s\,a_s^{-1}{\dot a}_s,
\end{equation}
with the natural meaning for the notation used; for example,  the second term on the right is the  vector  at the point ${\ovG}_s(t)a_s\in P$ arising from the vector $a_s^{-1}{\dot a}_s\in L(G)$. We note that the second term on the right is a vertical vector.  We recall that ${B_0}$ vanishes on vertical vectors and $A$, being a connection form, maps a vertical vector of the form $pZ\in T_pP$ to $Z\in L(G)$. 
Applying $\om$ to $ \partial_t{\tilde\Gamma}_s(t)$ we then obtain
\begin{equation}
\begin{split}
 & A\Bigl(\partial_t{\ovG}_s(t_0)a_s\Bigr) +{C^R_0}\Bigl(\partial_t{\ovG}_s(t_1)a_s\Bigr)-{C^L_0}\Bigl(\partial_t{\ovG}_s(t_0)a_s\Bigr)\\
 & \qquad + a_s^{-1}{\dot a}_s+ \int_{t_0}^{t_1}{B_0}\Bigl( \partial_s{\ovG}_s(t)a_s, \partial_t{\ovG}_s(t)a_s\Bigr)\,dt.
 \end{split}
\end{equation}
The condition that $s\mapsto {\tilde\Gamma}_s$ is $\om$-horizontal is that the above expression is $0$ for $s\in [s_0,s_1]$. Using the equivariance properties of $A$, ${B_0}$, ${C^L_0}$ and ${C^R_0}$, this condition is then equivalent to
\begin{equation}\label{E:adotsas}
\begin{split}
 {\dot a}_sa_s^{-1}& =-A\Bigl(\partial_t{\ovG}_s(t_0)\Bigr)-{C^R_0}\Bigl(\partial_t{\ovG}_s(t_1)\Bigr) + {C^L_0}\Bigl(\partial_t{\ovG}_s(t_0)\Bigr) \\
 &\hskip 1in - \int_{t_0}^{t_1}{B_0}\Bigl( \partial_s{\ovG}_s(t), \partial_t{\ovG}_s(t)\Bigr)\,dt.
  \end{split}
\end{equation}
Now the definition of $a_s$ given in (\ref{E:defasG}), as the `shift' that should be applied to ${\ovG}_s$ to yield ${\tilde\Gamma}_s$, shows that at $s=s_0$ the value is $e$ because, by our definition of $s\mapsto {\ovG}_s$ the initial path $\ovG_{s_0}$ is the same as the given initial path ${\tilde\Gamma}_{s_0}$. Since the right hand side of (\ref{E:adotsas}) involves only $C^\infty$ functions, there is a unique $C^\infty$ solution path
$$[s_0,s_1]\to G: s\mapsto a_s.$$
Thus we have   constructed  the $\om$-horizontal lift 
\begin{equation}\label{E:sGamtildGam}
[s_0,s_1]\to\pap: s\mapsto {\tilde\Gamma}_s=\ovG_sa_s
\end{equation}
 of the given path $s\mapsto\Gamma_s$ on $\mpm$.
 
 Since the ordinary differential equation  (\ref{E:adotsas}) has a unique solution with given initial value $a_{s_0}=e$ it follows that {\em any $C^\infty$ path  $s\mapsto \Gamma_s$ on $\mpm$ has a unique $\om$-horizontal lift to a path $s\mapsto {\tlG}_s$ on $\pap$ with given initial path $\tlG_0$.}

Let us note that if $A=\ovA$ then the first term on the right in (\ref{E:adotsas}) is $0$. No essential generality is achieved by working with $A$ instead of $\ovA$ because their difference could be absorbed into $C^L_0$.

%%%%%%%%%%%%%%%%%%%%%%%%%%%%%%%%%%%%%%%%%%%%%%%%%%%%%%%%%%%%%%%%%%%%%%%%%%%%%%%%%%%%%%%%%%%%%%%%%%%%%%%%%%%%%%%%%%%%%%%%%%%%%%%%%%%%%%%%%%%%%%%
\section{The decorated bundle}\label{s:decbun}

In this section we shall construct a `decorated' principal bundle over a path space starting with a traditional principal bundle along with some additional data. This notion was introduced in our earlier work \cite{CLS2geom} where we developed it from a mainly category-theoretic point of view. In this section we shall explore this notion from a  more differential geometric standpoint. Furthermore, we shall work out several formulas, such as for the derivatives of right actions on the relevant bundles, that will be of use later when we work with connection forms.

As we have remarked before, the motivation for constructing the decorated bundle comes from a physics context, where the decoration arises from a second structure group that  describes a gauge theory where point particles are replaced by paths.

\subsection{Lie crossed modules}
  
A {\em Lie crossed module} $(G, H, \alpha, \tau)$ is comprised of Lie groups $G$ and $H$, and homomorphisms
\begin{equation}\label{E:taualpha}
\tau:H\rightarrow G \quad\hbox{and}\quad \alpha:G\rightarrow {\rm Aut}(H),
\end{equation}
with $\tau$ being smooth and the map $G\times H\to H:(g,h)\mapsto\alpha(g)(h)$ also smooth, satisfying
\Beq\label{liecross}
\begin{split}
 \tau(\alpha(g)(h)) &=g\tau(h)g^{-1}, \qquad \forall g\in G, h\in H,\\
 \alpha(\tau(h))(h') &=hh'h^{-1}, \qquad \forall h, h'\in H.
\end{split}
\Eeq 
For  later use we note that the derivative of the first relation in (\ref{liecross}) leads to
\begin{equation}\label{E:liecross1der}
\tau\bigl[\alpha(g)X\bigr]={\rm Ad}(g)\tau(X) 
\end{equation}
for all $g\in G$ and $X\in L(H)$, and we have used the following natural notation: in (\ref{E:liecross1der})
   $\tau$ means $\tau(X)=d\tau|_eX$, and    $\alpha(g)X=d\alpha(g)|_eX$. 
   
   We will not need the map $\tau$ until we consider questions involving the composition of paths; in the present section as well as the next two sections we need only the semidirect product group $H\rtimes_{\alpha}G$ and not $\tau$.

\subsection{Lie crossed modules and conjugation}
Let $(G, H, \alpha, \tau)$ be a Lie-crossed module as defined in \eqref{liecross}. In the previous section we have constructed 
a principal $G$-bundle $ \pi_{\bar A}: {\mathcal P}_{\bar A}P\to \mpM$; more accurately, the mapping $ \pi_{\bar A}$ is surjective and we have shown a right action of $G$ on ${\mathcal P}_{\bar A}P$ which is free and transitive on each fiber. From this we will now construct  a principal  bundle whose structure group is the semidirect product $H\rtimes_{\alpha} G$; the product law in this group is given by
\begin{equation}\label{E:semidprodlaw}
(h_1, g_1)(h_2, g_2)=\bigl(h_1\alpha(g_1)(h_2), g_1g_2\bigr).
\end{equation}
Identifying $H$ and $G$ in the natural way as subgroups in $H\rtimes_{\alpha}G$ we have the commutation relation
\begin{equation}\label{E:semidcomm}
hg=g\alpha(g^{-1})(h)
\end{equation}
for all $h\in H$ and $g\in G$; to verify this note that the left side is, by definition, $(h,e)(e,g)=(h,g)$ and the right side is $(e,g)\bigl(\alpha(g^{-1})(h), e\bigr)$. The commutation relations can be used to reformulate some of our constructions below in a manner where the elements of $G$ appear before the elements of $H$ and for some relations this results in clearer expressions.  For example, the commutation relation is also equivalent to
\begin{equation}\label{E:semidcomm2}
gh= \alpha(g)(h)g.
\end{equation}

It is also very useful to note that after identifying $G$ and $H$ with subgroups of $H\rtimes_{\alpha} G$ (specifically, writing $g$ for $(e,g)$ and $h$ for $(h,e)$), the commutation relation gives the following friendly form for the automorphism $\alpha$:
\begin{equation}\label{E;alphaconj}
\alpha(g)(h)= ghg^{-1};
\end{equation}
thus  {\em the automorphism $\alpha(g)$ is simply conjugation} by $g$ in $H\rtimes_{\alpha} G$ restricted to the subgroup $H\simeq H\times\{e\}$.

The identification of $H$ and $G$ with the corresponding subgroups of $H\rtimes_{\alpha}G$ often provides a great simplification of notation. An  example of this simplification is seen in the derivative of the mapping $\alpha(g):H\to H$ at $e\in H$:
\begin{equation}\label{E:alphagAdg}
\alpha(g)={\rm Ad}(g):L(H\rtimes_{\alpha}G)\to L(H\rtimes_{\alpha}G).
\end{equation}

 \subsection{The decorated bundle} The total space of the bundle we will study is obtained by {\em decorating} $P$ with
elements of $H$: 
\Beq\label{dec}
{\mathcal P}^{\rm dec}_{\bar A}P:= {\mathcal P}_{\bar A}P\times H.
\Eeq
The  bundle projection is given by the map
\Beq\label{decproj}
\pi_{\rm dec}: {\mathcal P}^{\rm dec}_{\bar A}P\rightarrow \mpM: (\ovg, h)\mapsto \pi_{\bar A}(\ovg)=\pi\circ{ \ovg}.
\Eeq
The group $H\rtimes_{\alpha} G$ acts on the right on the  space ${\mathcal P}^{\rm dec}_{\bar A}P$ by 
\Beq\label{gract}
(\ovg, h)(h_1, g_1):=\Bigl(\ovg g_1, \alpha(g_1^{-1})(h h_1)\Bigr).
\Eeq
It will be notationally convenient to write $({ \overline\gamma}, h)$ as ${ \overline\gamma}h$; then the action (with a dot, which we shall later omit, for visual clarity) reads
\begin{equation}\label{E:gamhh1g1}
{ \overline\gamma}h\cdot h_1g_1=   { \overline\gamma}g_1\cdot \alpha(g_1^{-1})(hh_1),
\end{equation}
an expression which has a formal consistency with the commutation relation (\ref{E:semidcomm}). What we are denoting ${ \overline\gamma}h$ here is what is denoted $({ \overline\gamma}, h^{-1})$ in \cite{CLS2geom}. 
We note that for any ${ \overline\gamma}\in\pap$, $g\in G$ and $h\in H$,
\begin{equation*}
\begin{split}
\hbox{${ \overline\gamma}g$ is an element of $\pap$;}\\
\hbox{${ \overline\gamma}h$ is an element of $\pap\times H$.}\\
\end{split}
\end{equation*}
The notation
$$\ovg hg$$
might have two potentially different meanings:
$$\hbox{ $(\ovg e_H)\cdot (hg)=(\ovg, e_H)(h,g)$ or $(\ovg h)\cdot g=(\ovg, h)(e_H,g)$,}$$
 where we have written $e_H$ to stress that it is the identity in $H$.
However, we can readily verify the following
  notational consistencies:
\begin{equation}\label{E:notcons}
\begin{split}
{ \overline\gamma}e_H\cdot h &={ \overline\gamma}h\\
{ \overline\gamma}e_H\cdot hg &={ \overline\gamma}g\cdot\alpha(g^{-1})(h)={ \overline\gamma}h\cdot g,
\end{split}
\end{equation}
where on the right sides $h$ is, technically, $(h,e)$ and $g$ is $(e,g)$ in  $ H\rtimes_{\alpha}G$.

Because of the relations (\ref{E:notcons}
),  we can write ${ \overline\gamma}e_H$ simply as ${ \overline\gamma}$ if needed.

\begin{prop}\label{P:actionprops}
The mapping 
\begin{equation}\label{E:rightactsemid}
{\mathcal P}^{\rm dec}_{\bar A}P\times (H\rtimes_{\alpha}G)\to {\mathcal P}^{\rm dec}_{\bar A}P: ({ \overline\gamma}h, h_1g_1)\mapsto { \overline\gamma}h\cdot h_1g_1
\end{equation}
is a right action. This action is free and is transitive on the fibers of $\pi_{\ovA}^{\rm dec}:{\mathcal P}^{\rm dec}_{\bar A}P\to {\mathcal P} M$.
\end{prop}
\begin{proof} We verify that (\ref{E:rightactsemid}) specifies a right action:
\begin{equation}\label{E:ovghgrght}
\begin{split}
{ \overline\gamma}h\cdot (h_1g_1h_2g_2) &={ \overline\gamma}h\cdot \bigl(h_1\alpha(g_1)(h_2)g_1g_2 \bigr)\\
&={ \overline\gamma}g_1g_2\cdot\alpha(g_2^{-1}g_1^{-1})\bigl(hh_1\alpha(g_1)(h_2)\bigr)\\
&={ \overline\gamma}g_1g_2\cdot\alpha(g_2^{-1})\Bigl(\alpha(g_1^{-1})\bigl(hh_1\bigr) h_2\Bigr)\\
&=\Bigl({ \overline\gamma}g_1\cdot\alpha(g_1^{-1})(hh_1)\Bigr)\cdot h_2g_2\\
&=\bigl({ \overline\gamma}h\cdot h_1g_1\bigr)h_2g_2.
\end{split}
\end{equation}
Let us now verify that the action is free. Suppose
\begin{equation}\label{E:ggamh1g1free}
\hbox{a relation}\quad { \overline\gamma}h\cdot h_1g_1={ \overline\gamma}h.
\end{equation}
By (\ref{E:gamhh1g1}), this means
$${ \overline\gamma}g_1\cdot\alpha(g_1^{-1})(hh_1)={ \overline\gamma}h,$$
 which in turn is equivalent to
 $$g_1=e\qquad\hbox{and}\qquad hh_1=h$$
 with the latter relation being equivalent to $h_1=e$. Thus the fixed point relation (\ref{E:ggamh1g1free}) implies that $h_1g_1=e$.
 
 Next suppose $\pi_{\bar A}^{\rm dec}(\ovg_1, h_1)=\pi_{\bar A}^{\rm dec}(\ovg_2, h_2)$. Then ${ \overline\gamma}_1$ and ${ \overline\gamma}_2$, both  $\ovA$-horizontal paths in $P$, project down to the same path $\gamma\in\mpm$, and so there is a $g\in G$ such that
 $${ \overline\gamma}_2={ \overline\gamma}_1g.$$
 Next we observe that
\begin{equation}
\begin{split}
{ \overline\gamma}_2h_2 &= 
{ \overline\gamma}_1gh_2\\
&={ \overline\gamma}_1\alpha(g)(h_2)g \quad\hbox{(by the commutation relation (\ref{E:semidcomm2}))}\\
&={ \overline\gamma}_1h_1\cdot h_1^{-1}\alpha(g)(h_2)g.
 \end{split}
 \end{equation}
 Thus $({ \overline\gamma}_2,h_2)$ is obtained by acting on  $({ \overline\gamma}_1,h_1)$ with the element $h_1^{-1}\alpha(g)(h_2)g\in H\rtimes_{\alpha}G$.
 \end{proof}
 
 As an illustration of the power of working within the semidirect product and using the notation $hg=(h,g)$ and $\ovg hg =(\ovg, h)(e_H,g)$, we realizw that the computation (\ref{E:ovghgrght}) becomes trivial using this notation:
 \begin{equation}\label{E:ovghgrght2}
\begin{split}
{ \overline\gamma}h\cdot (h_1g_1h_2g_2) &={ \overline\gamma}h\cdot \bigl(h_1\alpha(g_1)(h_2)g_1g_2 \bigr)\\
&={ \overline\gamma}g_1g_2\cdot\alpha(g_2^{-1}g_1^{-1})\bigl(hh_1\alpha(g_1)(h_2)\bigr)\\
&={ \overline\gamma}g_1g_2\cdot\alpha(g_2^{-1})\Bigl(\alpha(g_1^{-1})\bigl(hh_1\bigr) h_2\Bigr)\\
&=\Bigl({ \overline\gamma}g_1\cdot\alpha(g_1^{-1})(hh_1)\Bigr)\cdot h_2g_2\\
&=\bigl({ \overline\gamma}h\cdot h_1g_1\bigr)h_2g_2.
\end{split}
\end{equation}
 
  \subsection{Local trivialization}
We have seen that a local trivialization 
$$\phi:U\times G\to\pi^{-1}(U)$$
 of the original bundle $\pi:P\to M$ leads to a local trivialization of   $(\pi_{\bar A}, {\mathcal P}_{\bar A}P, \mpM)$ given in (\ref{E:defphi0}) by
$$\phi^0:    U^0\times G\to \pi_{\bar A}^{-1}(U^0).   $$
Let $\phi_0$ be the inverse of $\phi^0$; thus,
$$\phi_0: \pi_{\bar A}^{-1}(U^0)\to   U^0\times G$$
is a $G$-equivariant bijection. Now we can construct  
  a local trivialization for the $(H\rtimes_{\alpha} G)$-bundle $(\pi_{\rm dec}, {\mathcal P}^{\rm dec}_{\bar A}P, \mpM)$ by means of the mapping:
  
\begin{equation}\label{dectriv}
\begin{split}
\phi^{\rm dec}: U^0\times (H\rtimes_{\alpha}G) &\to  \pi_{\rm dec}^{-1} (U^0)      \\
 \bigl(\gamma, (h,g)\bigr) &\mapsto  \bigl({ \overline\gamma},\alpha(g^{-1})(h)\bigr)\end{split}
\end{equation}
where ${ \overline\gamma}=\phi^0(\gamma, g)$ is the ${\bar A}$-horizontal lift of $\gamma$ starting at the  the point $\phi(\gamma_0,g)$, with $\gamma_0$ being the source (initial point) of $\gamma$.
The inverse of this is:
\begin{equation}\label{dectriv2}
\begin{split}
\phi_{\rm dec}:   \pi_{\rm dec}^{-1} (U^0)  &\to  U^0\times (H\rtimes_{\alpha}G) 
\\
({ \overline\gamma},h) &\mapsto \Bigl({\gamma}, \bigl(\alpha(g)(h), g\bigr)\Bigr)\\
&\quad\hbox{where $g$ is specified by $(\gamma, g)=\phi_0({ \overline\gamma}) $.}\end{split}
\end{equation}

We can check that this is equivariant under the action of $H\rtimes_{\alpha}G$:

\begin{equation}
\begin{split}
\phi^{\rm dec}\bigl(\gamma, hg)\cdot h_1g_1&=\bigl(\phi^0(\gamma, g), \alpha(g^{-1})(h)\bigr)h_1g_1 \\
&= \Bigl(\phi^0(\gamma, g)g_1, \alpha\bigl(g_1^{-1})(\alpha(g^{-1})(h)h_1\bigr)  \Bigr)\\
&=  \Bigl(\phi^0(\gamma, gg_1), \alpha(g_1^{-1}g^{-1})\bigl(h \alpha(g)(h_1)\bigr)  \Bigr)
\end{split}
\end{equation}
which agrees with
\begin{equation}\phi^{\rm dec}\bigl(\gamma, hgh_1g_1\bigr)= \Bigl(\phi^0\bigl(\gamma, gg_1\bigr),  \alpha\bigl((gg_1)^{-1}\bigr)(h\alpha(g)(h_1)\Bigr)\end{equation}
 
Thus we have proved:
\begin{prop}\label{pr:gbundbaradec}
 $(\pi_{\rm dec}, {\mathcal P}^{\rm dec}_{\bar A}P, \mpM)$ is a principal $H\rtimes_{\alpha}G$-bundle with right-action given by 
\eqref{gract} and  local trivialization given by \eqref{dectriv}.
\end{prop}

Let us note that when working with bundles over spaces of paths we do not use   a topology or an explicitly stated smooth structure.  This is discussed further in section \ref{s:diffcalc}. 

\subsection{The derivative of the right action} We turn now to some derivative computations that will be useful later, for example in Proposition \ref{pr:decconnect2}, when we study a connection form $\Omega$ on the bundle  of decorated paths. We view ${\mathcal P}^{\rm dec}_{\bar A}P=\pap\times H$ as we would  a product manifold. Thus we specify  a tangent vector ${\hat v}$
at $({\overline \gamma}, h)\in \pdbap$ by
 $$\hat v={\overline v}+X,$$
where ${\overline v} \in T_{{\overline\gamma}}\pap$ and $X\in T_{h}H$. (The tangent space $T_{{\overline\gamma}}\pap$ is itself identifiable with $T_{\gamma}\mpm\oplus L(G)$, where $\gamma=\pi\circ{\overline\gamma}$, by means of a local trivialization.) Thus we will   take the tangent space $T_{({ \overline\gamma},h)}\pdbap$ to be
\begin{equation}\label{E:tangpdec}
T_{({ \overline\gamma},h)}\pdbap=T_{{ \overline\gamma}}\pap\oplus T_hH.
\end{equation}
Recalling from (\ref{gract})  the right action of $H\rtimes_{\alpha}G$  on $\pdbap$ given by
\begin{equation}\label{E:tidgamhh1g1}
({ \overline\gamma},h)h_1g_1=\bigl({ \overline\gamma}g_1, \alpha(g_1^{-1})(hh_1)\bigr),
\end{equation}
we take, for fixed $(h_1,g_1)\in H\rtimes_{\alpha}G$,  the `differential' of the map
$$\pdbap\to\pdbap: ({ \overline\gamma},h)\mapsto({ \overline\gamma},h)h_1g_1$$
 to be given by  
\Beq\label{rightvect}
\begin{split}
  {\mathcal R}_{(h_1,g_1) *}:T_{({\overline \gamma}, h)}\pdbap &\rightarrow T_{({\overline \gamma}, h)h_1g_1}\pdbap\\
 {\mathcal R}_{(h_1, g_1) *}({\overline v}+X):&=({\overline v}+X)(h_1, g_1)={ \overline v} g_1+ g_1^{-1}(Xh_1)g_1,\\
\end{split}
\Eeq
 where $Xh_1\in T_{hh_1}H$  is the image of $X\in T_hH$ under the derivative of the right-translation map $H\to H:x\mapsto xh_1$, and the last term on the right hand side is, more precisely, the derivative $d\alpha(g)|_{hh_1}$ applied to $Xh_1$.

\subsection{Derivative of the orbit map} 
Next let us look at what should be taken to be the differential of the map
$$H\rtimes_{\alpha}G\to \pdbap: (h_1,g_1)\mapsto ({ \overline\gamma}, h)h_1g_1,$$
where $({ \overline\gamma},h)$ is any fixed point in $\pdbap$.  (This will be useful when we study the connection form $\Omega$ in Proposition   \ref{pr:decconnect2}.) We use the realization of the tangent space to $H\rtimes_{\alpha}G$ as
$$T_{(h_1,g_1)}(H\rtimes_{\alpha}G)= T_{h_1}H\oplus T_{g_1}G,$$
and write a vector in this space in the form 
$$h_1Y_1+g_1Z_1\stackrel{\rm def}{=} (h_1Y_1, g_1Z_1)\in T_{h_1}H\oplus T_{g_1}G,$$
where  $Y_1\in L(H)$ and $Z_1\in L(G)$.  Here, as always, $xV$ means the derivative of the left-translation map $G\to G:y\mapsto xy$ by $x$ on $V\in T_xG$ and $Vx$ has an analogous meaning with respect to right translations. We will often use notation  such as $xV$ that makes it possible to see at a glance that we are speaking of a vector located at the point $x$. 

 We also realize the tangent space $T_{({ \overline\gamma}, h)}(\pdbap)$ as
$$T_{({ \overline\gamma}, h)}(\pdbap)=T_{\ovg}(\pap) \oplus T_hH.$$

We will frequently need to use the derivative of the inversion map
$$j:G\to G:g\mapsto g^{-1},$$
and this is given by
\begin{equation}\label{E:dj}
dj|_g(W)= -g^{-1}Wg^{-1},
\end{equation}
for all tangent vectors $W\in T_gG$.  In particular if $W=gZ$, where $Z\in L(G)$, then
\begin{equation}\label{E:dj2}
dj|_g(gZ)=-Zg^{-1}.
\end{equation}
As always we denote by
$${\ovg}g_1Z$$
the vertical vector field along ${ \overline\gamma}g_1$ whose value at any parameter value $t$ is 
\begin{equation}\label{E:tildgammatZ}
{ \overline\gamma}(t)g_1Z=\frac{d}{ds}\Big|_{s=0}{ \overline\gamma}(t)g_1\exp(sZ).
 \end{equation}

 Let us write the right action of $H\rtimes_{\alpha}G$ on $\pap\times H$, as we have done in (\ref{gract}),   in the form
\begin{equation}\label{E:rightactconj}
({ \overline\gamma},h)h_1g_1=({ \overline\gamma}g_1, g_1^{-1}hh_1g_1).
\end{equation}
Holding $(\ovg, h)$ fixed,    the derivative of the orbit map
$$h_1g_1\mapsto (\ovg, h)h_1g_1$$
is given by
\begin{equation}\label{E:diggDgZ2}
 \begin{split}
 r_{({ \overline\gamma},h), (h_1,g_1)}:T_{(h_1,g_1)}(H\rtimes_{\alpha}G)  &\to  T_{( { \overline\gamma}, h)h_1g_1}\pdbap \\
 h_1Y_1+g_1Z_1  &\mapsto { \overline\gamma}g_1Z_1 +g_1^{-1}hh_1Y_1g_1\\
 &\hskip 1in + \Bigl(g_1^{-1}hh_1g_1 Z_1- Z_1g_1^{-1}hh_1g_1\Bigr),  
 \end{split}
 \end{equation}
 where the last expression, comprised of two terms within $(\cdots)$, lies in $T_{hh_1}H$, by the reasoning used below  in (\ref{E:diggDgZ3}). Let us note here the distinction between the derivative $ r_{({ \overline\gamma},h), (h_1,g_1)}$ and the derivative  ${\{\mathcal R}_{(h,g) *}$.  
  
 Here and in most of our computations we identify the Lie algebras of $H$ and $G$ with the corresponding subalgebras inside $L(H\rtimes_{\alpha}G)$. Thus
 \begin{equation}\label{E:lhgds}
 \hbox{$L(H\rtimes_{\alpha}G)= L(H)\oplus L(G)$ as a direct sum of {\em vector spaces.}}
 \end{equation}
 Evaluating the derivative in (\ref{E:diggDgZ2}) at the identity $(e,e)\in H\rtimes_{\alpha}G$ we obtain the linear map
\begin{equation}\label{E:diggDgZ3}
 \begin{split}
r_{({ \overline\gamma},h)}: L(H\rtimes_{\alpha}G)  &\to  T_{( { \overline\gamma}, h)}\pdbap \\
 Y_1+ Z_1  &\mapsto { \overline\gamma}Z_1 + h Y_1 +  h  Z_1- Z_1 h\\  
 &=  { \overline\gamma}Z_1 + h\Bigl( Y_1 +  \bigl(1-{\rm Ad}(h^{-1})\bigr) Z_1\Bigr), \end{split}
 \end{equation}
 where $Y_1\in L(H)$ and $Z_1\in L(G)$. Note that ${\rm Ad}(h^{-1})Z_1$ is obtained by applying  ${\rm Ad}(h^{-1})$ to $Z_1$, with everything taking place inside the Lie algebra $L(H\rtimes_{\alpha}G)$. The term $ \bigl(1-{\rm Ad}(h^{-1})\bigr) Z_1\Bigr)$ lies in $L(H)$, which can be seen by examining the derivative, at the identity in $G$,  of the mapping
  $$G\to H: g_1\mapsto g_1hg_1^{-1}=\alpha(g_1)(h).$$

%%%%%%%%%%%%%%%%%%%%%%%%%%%%%%%

\section{Connections   on the decorated bundle}\label{s:condecbun}

We continue with the framework established in the preceding sections. Thus
$\ovA$ is a connection on a principal $G$-bundle 
$$\pi:P\to M,$$
and $(G, H, \alpha,\tau)$ is a Lie crossed module. We have then a connection $\oab$,  constructed from a connection $A$ on $P$ and an $L(G)$-valued $2$-form ${B_0}$ on $P$,  on the horizontal path bundle
$$\pap\to\mpm,$$
which is also a principal $G$-bundle in the sense discussed before. We have constructed the decorated principal $H\rtimes_{\alpha}G$-bundle
$$\pdbap=\pap\times H\to\mpm.$$
In this section we shall   construct a connection on this decorated bundle by using the connection $\om$ and two additional forms on $P$ as ingredients.

\subsection{The   endpoint  forms $C_1^{L,R}$ and a $2$-form ${B_1}$} We shall work with  an  $L(H)$-valued  $2$-form ${B_1}$ on $P$ with following properties: 
\Beq\label{propD}
\begin{split}
&  {B_1}(ug, vg))={\rm Ad}(g^{-1})({B_1}(u, v))\qquad \forall u, v \in T_pP, g\in G, \\
&{B_1} (u, v)=0, \qquad {\rm if}\hskip 0.1 cm  u\hskip 0.1 cm {\rm or}\hskip 0.1 cm v \hskip 0.1 cm{\rm is \hskip 0.1 cm vertical}.
\end{split}
\Eeq    
Let us note here that in the first equation above, ${\rm Ad}(g)$ is acting on $ L(H\rtimes_{\alpha}G)$, and it is the same as $\alpha(g)$, as we have noted in (\ref{E:alphagAdg}).  We shall also use   $L(H)$-valued $1$-forms $C_1^L$ and $C_1^R$ on $P$ that have the following properties:
\Beq\label{propC}
\begin{split}
C_1^L|_{pg}(vg)  &={\rm Ad} (g^{-1}) C_1^L|_p(v) \qquad \forall  v \in T_pP, g\in G, \\
  C_1^L|_p(v) &=0, \qquad \hbox{if $v\in T_pP$ is any vertical vector,} \end{split}
\Eeq  
for all $p\in P$ and the corresponding properties for $C_1^{R}$. These conditions imply
\Beq\label{propC_2}
\begin{split}
 C_1^L|_{pg}(vg)  &={\rm Ad}(g^{-1}) C_1^L|_p(v) \qquad \forall  v \in T_pP, g\in G, \\
   C_1^L|_p(v) &=0, \qquad \hbox{if $v\in T_pP$ is any vertical vector.} \end{split}
\Eeq  
 
Let $\Sigma$ be the Maurer-Cartan form on $H$:
$$\Sigma_h(X)=h^{-1}X, \qquad \forall h\in H, \, X\in T_hH,$$
where on the right the notation signifies the action of the derivative of the left-translation map $h_1\mapsto h^{-1}h_1$.

\subsection{The connection form $\Omega$}
As noted before, the Lie algebra $L(H\rtimes_{\alpha}G)$ is the vector space   direct sum $L(H)\oplus L(G)$ (the Lie algebra structure on $L(H\rtimes_{\alpha}G)$ is not, however, obtained as a direct sum of Lie algebras).  As before we denote the evaluation map at the initial point by:
$${\rm ev}_{0}: \pdbap\to P: (\ovg,h)\mapsto \ovg(t_0),$$
where the domain of $\ovg$ is an interval $[t_0,t_1]$. Using a connection form $A$ on $P$, along with the $L(H\rtimes_{\alpha}G)$-valued $2$- and $1$-forms 
\begin{equation}\label{E:fullCLR}
\begin{split} 
B &=B_0+B_1\\
C^{L,R} &=C_0^{L,R}+C_1^{L,R},
\end{split}
\end{equation} 
 we define a  $1$-form $\Omega$ on
$\pdbap$, with values in    $L(H\rtimes_{\alpha}G)$, as follows:
\begin{equation}\label{E:defOmegagh1}
\begin{split}
& \Omega_{ {\overline\gamma}, h}:=\\
 &  {\rm Ad}(h^{-1})\left[{\rm ev}_{0}^*\bigl(A- C^L\bigr)|_{\ovg(t_0)}+{\rm ev}_{1}^*(C^R)|_{\ovg(t_1)}+\int_{
{ \overline\gamma}}B\right]  \\
 &\qquad +\Sigma_{h},
\end{split}
\end{equation}
where on the right we view $\Sigma$ as a form on $\pap\times H$ with the obvious pullback from the projection map onto $H$.
 Thus
\Beq\label{deccon}
\begin{split}
&  \Omega_{{\overline\gamma}, h}({\overline v}+X) \\
&={\rm Ad}(h^{-1})\left[\omega_{{\overline\gamma}}({\overline v})+C^R_{1}|_{\ovg(t_1)}({\overline v}(t_1))-C^L_{1}|_{\ovg(t_0)}({\overline v}(t_0)) \right. \\
&\left.  \hskip 1in+\int_{t_0}^{t_1}{B_1}|_{\ovg(t)}({\overline v}(t), {\overline \gamma}'(t))dt+Xh^{-1}\right],\\
\end{split}
\Eeq
where $ {\overline v}+X\in T_{({\overline\gamma}, h)}\pdbap$, with ${ \overline v}$   a vector field along the path ${ \overline\gamma}$ and $X\in T_hH$,  the $1$-form $\omega$ is as defined in \eqref{E:defomabcr}:
\Beq\label{conn2}
\begin{split}
\omega_{{\overline \gamma}}({\overline v}):&=A_{\ovg(t_0)}({\overline v}(t_0))+{C^R_0}|_{\ovg(t_1)}\bigl(v_{\ovg(t_1)}\bigr)-{C^L_0}|_{\ovg(t_0)}\bigl(v_{\ovg(t_0)}\bigr)\\
&\qquad\qquad +   \int_{t_0}^{t_1}{B_0}|_{\ovg(t)}\bigl({\overline v}(t), {\overline \gamma}'(t)\bigr)\,dt,
\end{split}
\Eeq  
for every path ${ \overline\gamma}:[t_0,t_1]\to P$ in $\pap$.

Let us recall   the right action
\begin{equation}
\pdbap\times (H\rtimes_{\alpha}G)\to \pdbap:\Bigl((\ovg, h), (h_1, g_1)\Bigr)\mapsto ({\ovg}g_1, g_1^{-1}hh_1g_1).
\end{equation}
From this map we have the two `directional derivatives':
\begin{equation}
\begin{split}
{\mathcal R}_{(h_1,g_1) *}: T_{\ovg, h}(\pdbap) &\to  T_{(\ovg , h) h_1g_1}(\pdbap)\\
&\hbox{and}\\
r_{({ \overline\gamma},h)}: T_{(h_1,g_1)}(H\rtimes_{\alpha}G) &\to T_{(\ovg , h) h_1g_1}(\pdbap).
\end{split}
\end{equation}

  \begin{prop}\label{pr:decconnect2} 
The $1$-form $\Omega$ is a connection on $(\pi_{\rm dec}, {\mathcal P}^{\rm dec}_{\bar A}P, \mpM)$ in the sense that the following conditions hold.
\begin{itemize}
\item[(i)]  It is equivariant under the right-action of $H\rtimes_{\alpha}G$:
\begin{equation}\label{E:Omrighteqv}
\Omega_{({ \overline\gamma},h)h_1g_1} {\mathcal R}_{(h_1,g_1) *}({ \overline v},X) =
{\rm Ad}(h_1g_1)^{-1}\Omega_{({ \overline\gamma},h)} ({ \overline v},X) 
\end{equation}
for all $ h,h_1\in H$, $g_1\in G$, ${ \overline\gamma}\in\pap$, ${ \overline v}\in T_{{ \overline\gamma}}\pap$ and $X\in T_hH$;
\item[(ii)] It returns the appropriate elements of $L(H\rtimes_{\alpha}G)$ when applied to vertical vectors:
\begin{equation}\label{E:Omegvert}
\Omega_{({ \overline\gamma},h)} r_{({ \overline\gamma},h)}
(Y_1+Z_1)=Y_1+Z_1,
\end{equation}
where  $r_{({ \overline\gamma},h)}$ is the derivative of the right-action of $H\rtimes_{\alpha}G$ on $\pap$ as in (\ref{E:diggDgZ2}) and $(Y_1,Z_1)\in L(H)\oplus L(G)$.
\end{itemize}
\end{prop}
\begin{proof} For notational convenience we shall write
${\ovv}_0$ for ${\ovv}(t_0)$, $\ovg_0$ for $\ovg(t_0)$, and analogously for other paths and vector fields.

Working through the right-action we have
\begin{equation}\label{E:rightactcomput1}
\begin{split}
&\Omega_{({ \overline\gamma},h)h_1g_1} {\mathcal R}_{(h_1,g_1) *}({ \overline v},X)\\
 &=\Omega_{({ \overline\gamma}g_1, g_1^{-1}hh_1g_1)}({ \overline v}g_1+g_1^{-1}Xh_1g_1)\\
&\hskip 1in \hbox{using (\ref{rightvect})} \\
 &={\rm Ad}(g_1^{-1}h_1^{-1}h^{-1}g_1)\Bigl(\omega_{{ \overline\gamma}}({ \overline v}g_1)+C^R_1|_{{\ovg}_1g_1}({\ovv}_0g_1) -C^L_1|_{{\ovg}_0g_1}({\ovg}_0g_1) \Bigr)\\
 &\qquad
+ (g_1^{-1}hh_1g_1)^{-1}g_1^{-1}Xh_1 g_1 +{\rm Ad}(g_1^{-1}h_1^{-1}h^{-1}g_1)\int_{t_0}^{t_1}{B_1}\bigl({ \overline v}(t)g_1, { \overline\gamma}'(t)g_1\bigr)\,dt.
\end{split}
\end{equation}
We work out the  last term on the right separately:
\begin{equation}
\begin{split}
&{\rm Ad}(g_1^{-1}h_1^{-1}h^{-1}g_1)\int_{t_0}^{t_1}{B_1}\bigl({ \overline v}(t)g_1, { \overline\gamma}'(t)g_1\bigr)\,dt \\
&= {\rm Ad}(g_1^{-1}h_1^{-1}h^{-1}g_1){\rm Ad}(g_1^{-1}) \int_{t_0}^{t_1}{B_1}\bigl({ \overline v}(t), { \overline\gamma}'(t)\bigr)\,dt \\
&\qquad\hbox{(using the  equivariance property of ${B_1}$ from (\ref{propD}))}\\
&=  {\rm Ad}(g_1^{-1}h_1^{-1}h^{-1})\int_{t_0}^{t_1}D\bigl({ \overline v}(t), { \overline\gamma}'(t)\bigr)\,dt.
\end{split}
\end{equation}
Returning to our computation (\ref{E:rightactcomput1}), we have:
\begin{equation}\label{E:rightactcomput2}
\begin{split}
&\Omega_{({ \overline\gamma},h)h_1g_1} {\mathcal R}_{(h_1,g_1) *}({ \overline v},X)\\
&={\rm Ad}(g_1^{-1}h_1^{-1}h^{-1}g_1){\rm Ad}(g_1^{-1})\Bigl(\omega_{{ \overline\gamma}}({ \overline v})+ C^R_1|_{{\ovg}_1}({\ovv}_0) -C^L_1|_{{\ovg}_0}({\ovv}_0)\Bigr)
\\
&\qquad  
 +g_1^{-1}h_1^{-1}h^{-1}g_1\,g_1^{-1} Xh_1g_1+ {\rm Ad}(g_1^{-1}h_1^{-1}h^{-1})\int_{t_0}^{t_1}{B_1}\bigl({ \overline v}(t), { \overline\gamma}'(t)\bigr)\,dt.
\end{split}
\end{equation}
On the other hand
\begin{equation}\label{E:rightactcomput3}
\begin{split}
&{\rm Ad}(h_1g_1)^{-1}\Omega_{({ \overline\gamma},h)}  ({ \overline v},X)\\
&={\rm Ad}(h_1g_1)^{-1}{\rm Ad}(h^{-1})\Bigl(\omega_{{ \overline\gamma}}({ \overline v})+ C^R_1|_{{\ovg}_1}({\ovv}_0) -C^L_1|_{{\ovg}_0}({\ovv}_0) \Bigr)
 \\
&\qquad\qquad +{\rm Ad}(h_1g_1)^{-1}(h^{-1}X)
+{\rm Ad}(h_1g_1)^{-1}{\rm Ad}(h^{-1})\int_{t_0}^{t_1}{B_1}\bigl({ \overline v}(t), { \overline\gamma}'(t)\bigr)\,dt. 
\end{split}
\end{equation}
We see that this is equal to the expression on the right in (\ref{E:rightactcomput2}). This proves property (i) for $\Omega$.

Next we consider how the connection form acts on a `vertical vector' in the bundle $\pdbap$; such a vector is of the form
$$ r_{({ \overline\gamma},h)}
(Y_1+Z_1)\in T_{({ \overline\gamma},h)}(\pdbap),$$
where 
$$Y_1+Z_1=(Y_1,Z_1)\in L(H)\oplus L(G)),$$
is an arbitrary vector in $L(H\rtimes_{\alpha}G)$. Thus
\begin{equation}\label{E:rightactcomput4}
\begin{split}
&\Omega_{({ \overline\gamma},h)} r_{({ \overline\gamma},h)}
(Y_1+Z_1)\\
&=\Omega_{({ \overline\gamma},h)}\Bigl({ \overline\gamma}Z_1+h(Y_1+\bigl(1-{\rm Ad}(h^{-1})\bigr)Z_1\Bigr)\\
&\qquad\hbox{(using the expression for $r_{({ \overline\gamma},h)}$ obtained in (\ref{E:diggDgZ3}))}\\
&= {\rm Ad}(h^{-1})\Bigl(\omega_{{ \overline\gamma}}({ \overline\gamma}Z_1) +C^R_1|_{{\ovg}_1}({\ovg}_1Z_1)-C^L_1|_{{\ovg}_0}({\ovg}_0Z_1)\Bigr) \\
& \hskip 1in  +\Bigl[Y_1+\bigl(1-{\rm Ad}(h^{-1})\bigr)Z_1\Bigr]+ {\rm Ad}(h^{-1})\int_{t_0}^{t_1}{B_1}\bigl({ \overline\gamma}(t)Z_1, { \overline\gamma}'(t) \bigr)\,dt.
  \end{split}
\end{equation}
In the expression on the right, the  terms with $C$ and the last term, with ${B_1}$,  are $0$ because ${B_1}$, $C^L_1$ and  $C^R_1$      vanish on vertical vectors (see (\ref{propD}) and (\ref{propC})). The first term equals $Z_1$:
\begin{equation}\label{E:omegaZ}
\omega_{{ \overline\gamma}}({ \overline\gamma}Z_1)=Z_1,
\end{equation}
 as seen in (\ref{E:oABvert2}). 

Putting all this together we have
\begin{equation}\label{E:rightactcomput5}
\begin{split}
&\Omega_{({ \overline\gamma},h)} r_{({ \overline\gamma},h)}
(Y_1+Z_1)\\
&=  {\rm Ad}(h^{-1})Z_1+\Bigl[Y_1+\bigl(1-{\rm Ad}(h^{-1})\bigr)Z_1\Bigr] +0\\
&=Y_1+Z_1.
\end{split}
\end{equation} 
This proves property (ii). 
\end{proof}

The conditions (i) and (ii) above imply that the form $\Omega$ splits each tangent space  $T_{( \overline \gamma, h)}\pdbap$ into horizontal and vertical subspaces as explained in the following result.

\subsection{Horizontal and vertical parts} We turn now to understanding how the connection form $\Omega$ splits a vector $v\in T_{( \overline\gamma, h)}\pdbap$ splits into a horizontal and a vertical component.
\begin{prop}\label{pr:decconnect} At any 
$( \overline\gamma, h)\in \pdbap$, $\Omega$ splits $T_{( \overline \gamma, h)}\pdbap$ into a direct sum:
\begin{equation}\label{E:Tangdirectsum}
T_{( \overline\gamma, h)}\pdbap={\rm H}_{( \overline\gamma, h)}\pdbap\oplus {\rm V}_{({ \overline\gamma}, h)}\pdbap,
\end{equation}
where the `horizontal subspace' ${\rm H}_{({ \overline\gamma}, h)}\pdbap$ is 
\begin{equation}\label{horaversplit}
 {\rm H}_{({ \overline\gamma}, h)}\pdbap=\ker\Omega_{({ \overline\gamma},h)}, 
 \end{equation}
and the `vertical subspace' ${\rm V}_{({ \overline\gamma}, h)}\pdbap$ is the image of $r_{({ \overline\gamma},h)}$:
 \begin{equation}\label{E:diggDgZ4} {\rm V}_{({ \overline\gamma}, h)}\pdbap= \{r_{({ \overline\gamma},h)}(Y_1+Z_1): Y_1\in L(H), Z_1\in L(G)\}, \end{equation}
 as  noted in (\ref{E:diggDgZ3}), 
 with $r_{({ \overline\gamma},h)}$ being the right action of $H\rtimes_{\alpha}G$ on $\pdbap$.
\end{prop}

\begin{proof}  
Let 
$$\hat v={ \overline v}+X\in T_{({ \overline\gamma}, h)}\pdbap,$$
 where ${ \overline v} \in T_{{ \overline\gamma}}\pap, X\in T_h H$.  Let ${ \overline v}^{\rm H}$ and ${ \overline v} ^{\rm V}$ 
be, respectively, the horizontal and vertical components of ${ \overline v}$ with respect to the connection $\omega=\oab$.  Thus, in particular,
$\omega_{\ovg}({ \overline v}^{\rm H})=0$.

Let $\ovg_1$ denote the right endpoint $\ovg(t_1)$, and $\ovg_0$ denote the left endpoint $\ovg(t_0)$. We will now show that the horizontal and vertical components of $\hat v=\ovv+X$ are  
\begin{equation}\label{horizvertdec}
\begin{split}
{\hat v}^{\rm H} &= { \overline v}^{\rm H}-\Bigl(C^R_1|_{\ovg_1}\bigl(\ovv^{\rm H}(t_1)\bigr)-C^L_1|_{\ovg_0}\bigl(\ovv^{\rm H}(t_0)\bigr) +   \int_{t_0}^{t_1} B_1({ \overline v}^{\rm H}(t),{ \overline\gamma}'(t) )dt\Bigr)h
\\
 {\hat v}^{\rm V} &={ \overline v}^{\rm V}+X+\Bigl(C^R_1|_{\ovg_1}\bigl(\ovv^{\rm H}(t_1)\bigr)-C^L_1|_{\ovg_0}\bigl(\ovv^{\rm H}(t_0)\bigr)+   \int_{t_0}^{t_1} {B_1}({ \overline v}^{\rm H}(t),{ \overline\gamma}'(t) )dt\Bigr)h
,
 \end{split}
 \end{equation}
 respectively, where, as always, $Zh\in T_hH$ is the result of applying the derivative of the right translation map $H\to H:x\mapsto xh$ to $Z\in L(H)$.
 
  Our objective now is to show that ${\hat v}^{\rm H}$ lies in the horizontal subspace, $ {\hat v}^{\rm V}$ in the vertical subspace. 
  
   The relation (\ref{E:rightactcomput5})  shows that when $\Omega$ is applied to a vertical vector, which, by definition, is of the form  $v= r_{({ \overline\gamma},h)}(Y_1+Z_1)$ then the value obtained is $Y_1+Z_1$; hence if this is $0$  then $v$ itself is $0$. Thus  the only vertical vector which is also horizontal is just the zero vector. Thus   the sum in 
(\ref{E:Tangdirectsum}) is indeed a direct sum.

Inserting \eqref{horizvertdec} in the expression for $\Omega$ given in \eqref{deccon} we get
\begin{equation}
\begin{split}
{\rm Ad}(h) \Omega_{{ \overline\gamma}, h}({\hat v}^{\rm H}) & = \omega_{{ \overline\gamma}}({\overline v}^{\rm H})  
 +C^R_1|_{\ovg_1}\bigl(\ovv^{\rm H}(t_1)\bigr)-C^L_1|_{\ovg_0}\bigl(\ovv^{\rm H}(t_0)\bigr)\\
 &\qquad + \int_{t_0}^{t_1}{B_1}(\ovv^{\rm H}(t), { \overline\gamma}'(t))dt\\
& \qquad -   \Bigl(C^R_1|_{\ovg_1}\bigl(\ovv^{\rm H}(t_1)\bigr)-C^L_1|_{\ovg_0}\bigl(\ovv^{\rm H}(t_0)\bigr) +   \int_{t_0}^{t_1} {B_1}({ \overline v}^{\rm H}(t),{ \overline\gamma}'(t) )dt\Bigr)  \\
& =0.
\end{split}
\end{equation} 

The vector ${ \overline v}^{\rm V}$, which is the $\omega$-vertical part of ${ \overline v}$,  is given by
\begin{equation}\label{E:vVert}
 { \overline v}^{\rm V}= { \overline\gamma}Z,
 \end{equation}
 where 
 $$Z=\omega_{ \overline\gamma} ({ \overline v})$$
 (as we have discussed earlier in (\ref{E:oABvert})).
 Applying $\Omega_{{ \overline\gamma}, h}$ to ${ \overline v}^{\rm V}$ we have
\begin{equation}\label{E:OghvV}
\begin{split}
&\Omega_{{ \overline\gamma}, h}({\hat v}^{\rm V}) \\
&={\rm Ad}(h^{-1}) \left[Z+ Xh^{-1} +C^R_1|_{\ovg_1}\bigl(\ovv^{\rm H}(t_1)\bigr)-C^L_1|_{\ovg_0}\bigl(\ovv^{\rm H}(t_0)\bigr)\right. \\
&\hskip 2in \left.+   \int_{t_0}^{t_1} {B_1}({ \overline v}^{\rm H}(t),{ \overline\gamma}'(t) )dt \right].
\end{split}
\end{equation}
 We can write the right hand side of (\ref{E:OghvV}) as a sum of a vector in $L(H)$ and a vector in $L(G)$ on using the observation, made earlier after (\ref{E:diggDgZ3}), that $\bigl(1-{\rm Ad}(h^{-1})\bigr)Z$ is in $L(H)$. Thus we have
\begin{equation}
\begin{split}
&\Omega_{{ \overline\gamma}, h}({\hat v}^{\rm V})=\Bigl({\rm Ad}(h^{-1})-1\Bigr)Z \\
&\quad +h^{-1}\left[X+ \left(C^R_1|_{\ovg_1}\bigl(\ovv^{\rm H}(t_1)\bigr)-C^L_1|_{\ovg_0}\bigl(\ovv^{\rm H}(t_0)\bigr)  \right.\right.\\
&\left.\left.\hskip 1.25in + \int_{t_0}^{t_1} {B_1}(\ovv^{\rm H}(t),{ \overline\gamma}'(t) )dt\right)h \right] \\
&\qquad\qquad +Z.
 \end{split}
\end{equation}
Now we recall from (\ref{E:diggDgZ3})  that the derivative of the orbit map
$$H\rtimes_{\alpha}G\to \pdbap: (h_1,g_1)\mapsto ({ \overline\gamma},h)h_1g_1$$
at the identity element $(e,e)\in H\rtimes_{\alpha}G$ 
is given by
 \begin{equation}\label{E:diggDgZ5}
 \begin{split}
r_{({\ovg},h)}: L(H\rtimes_{\alpha}G)  &\to  T_{({ \overline\gamma}, h)}\pdbap \\
 Y_1+ Z_1  &\mapsto    { \overline\gamma}Z_1 + h\Bigl( Y_1 +  \bigl(1-{\rm Ad}(h^{-1})\bigr) Z_1\Bigr). \end{split}
 \end{equation}
 Applying this to $\Omega_{{ \overline\gamma}, h}({\hat v}^{\rm V})$ as given above we obtain
 \begin{equation}
\begin{split}
&r_{({ \overline\gamma},h)}\Bigl(\Omega_{{ \overline\gamma}, h}({\hat v}^{\rm V}) \Bigr)\\
&={ \overline\gamma}Z + h\Bigl({\rm Ad}(h^{-1})-1\Bigr)Z\\
&\qquad +\left[X+ \left(C^R_1|_{\ovg_1}\bigl(\ovv^{\rm H}(t_1)\bigr)-C^L_1|_{\ovg_0}\bigl(\ovv^{\rm H}(t_0)\bigr)  +\int_{t_0}^{t_1} {B_1}(\ovv^{\rm H}(t),{ \overline\gamma}'(t) )dt\right)h \right]  \\
&\hskip 1in +h\bigl(1-{\rm Ad}(h^{-1})\bigr)Z\\
&= { \overline\gamma}Z  + \left[X+ \left(C^R_1|_{\ovg_1}\bigl(\ovv^{\rm H}(t_1)\bigr)-C^L_1|_{\ovg_0}\bigl(\ovv^{\rm H}(t_0)\bigr)  \right.\right. \\
&\hskip 1in \left.\left.+\int_{t_0}^{t_1} {B_1}(\ovv^{\rm H}(t),{ \overline\gamma}'(t) )dt\right)h \right],  \end{split}
 \end{equation}
 which we recognize to be ${\hat v}^{\rm V}$ as given in (\ref{horizvertdec}).  Thus,
 $${\hat v}^{\rm V}=r_{({ \overline\gamma},h)}\Bigl(\Omega_{{ \overline\gamma}, h}({\hat v}^{\rm V}) \Bigr)$$
lies in the vertical subspace  ${\rm V}_{({ \overline\gamma}, h)}\pdbap$.
 \end{proof}
  
\section{Horizontal lifts of paths on decorated bundles}\label{s:horliftdec}

We work with the framework from the preceding sections, with a connection $\bar A$ on a principal $G$-bundle $\pi:P\to M$, and a Lie crossed module  $(G,H,\alpha,\tau)$; as noted before, we shall only use the semidirect product $H\rtimes_{\alpha}G$ and not $\tau$ at this stage.  There is then a principal $G$-bundle $\pap\to \mpm$, where the elements of $\pap$ are $\bar A$-horizontal paths on $P$. We have introduced a connection $\omega$ on   $\pap\to \mpm$, and a connection $\Omega$ on the
principal $H\rtimes_{\alpha} G$-bundle  $\pdbap\to \mpm$. The elements of $\pdbap$ are of the form $({ \overline\gamma}, h)$, where ${ \overline\gamma}\in\pap$  and $h\in H$. Our goal in this section is to determine parallel-transport, illustrated in Figure \ref{fig:pardeco}, by the connection $\Omega$. 

%%%%%%%%%%%%%%%%%%%%%%FIGURE HERE%%%%%%%%%%%%%%%%%%%%%%%%%%%%%%%%%%%%%
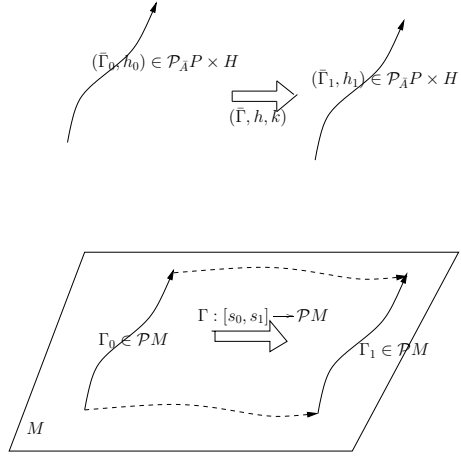
\begin{figure}[htbp]
\begin{center}
\resizebox{6cm}{!}{\input{higherpar.pstex_t}}
\caption{Parallel transport of decorated paths}
\label{fig:pardeco}
\end{center}
\end{figure}
%%%%%%%%%%%%%%%%%%%%%%%%%%%%%%%%%%%%%%%%%%%%%%%%%%%%%%%%%%%%%%%%%%%%%%

We have, as ingredients of the framework, the connection forms $\ovA$ and $A$ on $P$, the $L(H)$-valued $2$-form $B$ on $P$, the $L(H\rtimes_{\alpha} G)$-valued $1$-forms $C^{L,R}$  on $P$ and the $L(H\rtimes_{\alpha} G)$-valued $2$-form ${B}$ on $P$.  The forms $B$, $C$ and $D$ satisfy equivariance properties as discussed before and vanish when contracted with any vertical vector.  

We consider a path
$$\Gamma: [s_0,s_1]\to \mpm: s\mapsto\Gamma_s$$
where each $\Gamma_s$ is a $C^\infty$  path $[t_0,t_1]\to M$, for some $t_0<t_1$.  
We assume that $\Gamma$ is smooth in the sense that
\begin{equation}\label{E:Gamma}
[t_0,t_1]\times [s_0,s_1]\to M: (t,s)\mapsto \Gamma(t,s)=\Gamma_s(t)
\end{equation}
is $C^\infty$.  There are some additional technical requirements we impose in order to ensure that composition of paths of paths lead to  paths of paths of the same nature. To this end we assume that for the mapping $\Gamma$ there exists an $\epsilon>0 $  such that for each fixed $s$ the point
$$\Gamma_s(t)$$
remains constant when $t$ is within distance $\epsilon$ of $t_0$ or $t_1$, and for each fixed $t\in [t_0, t_1]$ the point $\Gamma_s(t)$ remains constant when $s$ is within distance $\epsilon$ of $s_0$ or $s_1$. Furthermore, we identify $\Gamma$ with the mapping
$$\Gamma^{-v}: ([t_0,t_1]\times [s_0,s_1])+v\to M: (t,s)\mapsto \Gamma\bigl((t,s)-v\bigr),$$
for any fixed $v\in\mbr^2$. More precisely, identification means that we form a quotient space ${\mathcal P}_2(M)$, where $\Gamma$ and $\Gamma^{-v}$ correspond to the same element. This identification makes it possible to speak of composition of two paths of the form $[s_0,s_1]\to\mpm:s\mapsto\Gamma_s$.

 Our goal is to determine the $\Omega$-horizontal lift of $s\mapsto \Gamma_s$, with a given initial point
$$({ \tlG}_{s_0},h_{s_0})\in {\pi_{\bar A}^{\rm dec}}^{-1}(\Gamma_{s_0}).$$

To this end let 
\begin{equation}
[s_0,s_1]\to \pap: s\mapsto { \tlG}_s
\end{equation}
be the $\omega$-horizontal lift of the path $s\mapsto \Gamma_s$, with initial point ${ \tlG}_{s_0}$.  (Recall that $\omega $ is a connection on $\pap$ and in subsection   \ref{ss:omhortr} we have shown the existence of $\om$-horizontal lifts.) Next let
\begin{equation}\label{E:hstauks}
s\mapsto h_s
\end{equation}
be the solution of the differential equation
\begin{equation}
 {\dot h}_sh_s^{-1}=- \left[C^R_1\bigl(\partial_s\tlG_s(t_1)\bigr)-C^L_1\bigl(\partial_s\tlG_s(t_0)\bigr)+\int_{t_0}^{t_1}{B_1}\Bigl(\partial_s{ \tlG}_s(t), \partial_t{ \tlG}_s(t)\Bigr)\,dt\right]
\end{equation}
with an initial value $h_{s_0}=e\in H$.

We recall our assumptions that $ C_1^{L,R}$ and $  {B_1}$ take values in the Lie algebra $L(H)\subset L(H\rtimes_{\alpha}G)$. As a result, $h_s$ lies in $H$.

We note that
\begin{equation}\label{E:dothdot}
\begin{split}
{\dot h}_sh_s^{-1}  
&=- \left[C^R_1\bigl(\partial_s\tlG_s(t_1)\bigr)-C^L_1\bigl(\partial_s\tlG_s(t_0)\bigr)+\int_{t_0}^{t_1}{B_1}\Bigl(\partial_s{ \tlG}_s(t), \partial_t{ \tlG}_s(t)\Bigr)\,dt\right].
\end{split}
\end{equation}  
Let us recall from (\ref{deccon}) the connection form $\Omega$ on $\pdbap$ given by:
\Beq\label{decOm}
\begin{split}
&  \Omega_{{\overline\gamma}, h}({\overline v}+X) \\
&={\rm Ad}(h^{-1})\left[\omega_{{\overline\gamma}}({\overline v})+ C^R_1\bigl(\partial_s\tlG_s(t_1)\bigr)- C^L_1\bigl(\partial_s\tlG_s(t_0)\bigr)\right.\\
&\qquad\left. \qquad \qquad\qquad+   \int_{t_0}^{t_1}{B_1}|_{\ovg(t_0)}({\overline v}(t), {\overline \gamma}'(t))dt+Xh^{-1}\right],\\
\end{split}
\Eeq
where $ {\overline v}+X\in T_{({\overline\gamma}, h)}\pdbap$, with ${ \overline v}$   a vector field along the path ${ \overline\gamma}:[t_0,t_1]\to P$ belong to  $\pap$ and $X\in T_hH$,  the $1$-form $\omega$ is as defined in \eqref{conn}:
\Beq\label{conn3}
\omega_{{\overline \gamma}}({\overline v}):=  A_{\ovg(t_0)}({\overline v}(t_0))+{C^R_0}({\overline v}(t_1))-{C^L_0}({\overline v}(t_0))+  \int_{t_0}^{t_1}B_0|_{\ovg(t_0)}\bigl({\overline v}(t), {\overline \gamma}'(t)\bigr)\,dt.
\Eeq

\begin{prop}\label{P:Omegahorlift} Suppose  $(G, H, \alpha,\tau)$ is a  Lie crossed module  and $\ovA$ is a connection on a principal $G$-bundle $\pi:P\to M$. We have then as above the bundle $\pap\to\mpm$ of $\ovA$-horizontal paths on $M$ over the space $\mpm$ of paths on $M$, and the decorated bundle
$$\pdbap=\pap\times H\to\mpm,$$
equipped with a connection form $\Omega$ given above in (\ref{decOm}), involving the forms $C_1^{R,L}$ and ${B_1}$ that take values in   $L(H)$. Then  the path
\begin{equation}
[s_0,s_1]\to  \pdbap: s\mapsto ({ \tlG}_s,h_s)
\end{equation}
is $\Omega$-horizontal if and only if $s\mapsto {\tlG}_s$ is an $\om$-horizontal path on $\pap$ and $s\mapsto h_s$ satisfies the differential equation
\begin{equation}\label{E:dothdot2}
{\dot h}_sh_s^{-1}=-\left[C^R_1\bigl(\partial_s\tlG_s(t_1)\bigr)- C^L_1\bigl(\partial_s\tlG_s(t_0)\bigr)+\int_{t_0}^{t_1}  {B_1}\Bigl(\partial_s{ \tlG}_s(t), \partial_t{ \tlG}_s(t)\Bigr)\,dt\right].
\end{equation}
\end{prop}

\begin{proof}
  Evaluating $\Omega$ on the tangent vector (field) 
 $$\partial_s ({ \overline\Gamma}_s,h_s)= \Bigl(\partial_s{ \tlG}_s, {\dot h}_s\Bigr)\in T_{({ \tlG}_s, h_s)}\pdbap,$$
 we have
\begin{equation}
\begin{split}
 &\Omega_{({ \tlG}_s, h_s)}\partial_s ({ \tlG}_s, h_s) ={\rm Ad}(h_s^{-1})\omega\bigl(\partial_s { \tlG}_s\bigr)   + {\rm Ad}(h_s^{-1})\left[{\dot h}_sh_s^{-1}  \right. \\
& \qquad\left. +  C^R_1\bigl(\partial_s\tlG_s(t_1)\bigr)-C^L_1\bigl(\partial_s\tlG_s(t_0)\bigr) +\int_{t_0}^{t_1}{B_1}\Bigl(\partial_s{ \tlG}_s(t), \partial_t{ \tlG}_s(t)\Bigr)\,dt\Bigr) \right].\end{split}
\end{equation}
Here, on the right, the first term is in $L(G)$ and the second term is in $L(H)$. The entire expression is $0$ if and only if each of these terms is $0$. This is equivalent to $s\mapsto {\tlG}_s$ being $\omega$-horizontal and $s\mapsto h_s$ satisfying the differential equation (\ref{E:dothdot2}). \end{proof}

\section{Passing to a higher decoration}\label{s:catgeom}

In the preceding sections we described a connection form $\oab$ on the principal $G$-bundle $\pap\to\mpm$, where $\pap$ is a space of $\ovA$-horizontal paths with $\ovA$ being a connection on a principal $G$-bundle $\pi:P\to M$. The form $\oab$, as  specified in (\ref{conn}), is given by  
\Beq\label{conn3ab}
\omega^{(A, B_0)}_{{ \overline\gamma}}({\overline v}):=A({\overline v}(t_0))+ \int_{t_0}^{t_1}{B_0}|_{\ovg(t)}({\overline v}(t), {{ \overline\gamma}}'(t))dt,
\Eeq  
where $A$ is a connection on $\pi:P\to M$ and $B$ is an $L(G)$-valued ${\rm Ad}$-equivariant $2$-form on $P$ that vanishes on vertical vectors. A more general connection on $\pap$ is $\om$, given by
\Beq\label{connom4}
\omega_{{ \overline\gamma}}({\overline v}):=\omega^{(A, B_0)}_{{ \overline\gamma}}({\overline v}) +{C^R_0}|_{\ovg(t_1)}\bigl(\ovv(t_1)\bigr)-{C^L_0}|_{\ovg(t_0)}\bigl(\ovv(t_0)\bigr),
\Eeq  
where $C^{L,R}_0$ are $L(G)$-valued $1$-forms on $P$ that vanish on vertical vectors and are Ad-equivariant.
 Next, moving up a level, we considered a     principal $H\rtimes_{\alpha}G$-bundle
$$\pdbap=\pap\times H\to \mpm,$$
which we call a {\em decorated} bundle. We brought in $L(H)$-valued $1$-form $C$ and an $L(H)$-valued $2$-form ${B_1}$ on $P$, both appropriately equivariant and contracting to $0$ on vertical vectors, and imposed the further condition (for convenience) that $ C$ and $  {B_1}$ take values in $L(H)$ rather than the full Lie algebra $L(H\rtimes_{\alpha}G)$. 
Then we introduced the connection form  $\Omega$ on $\pdbap$, specified at any point $({ \overline\gamma},h)\in\pdbap$ by
\Beq\label{decOm2}
\begin{split}
&  \Omega_{{\overline\gamma}, h}({\overline v}+X) \\
&={\rm Ad}(h^{-1})\left[\omega_{{\overline\gamma}}({\overline v})+ C^R_1|_{\ovg(t_1)}({\overline v}(t_1))- C^L_1|_{\ovg(t_0)}({\overline v}(t_0)) \right. \\
&\left. \hskip 1in +   \int_{t_0}^{t_1}{B_1}|_{\ovg(t)}({\overline v}(t), {\overline \gamma}'(t))dt+Xh^{-1}\right],\\
\end{split}
\Eeq
where $ {\overline v}+X\in T_{({\overline\gamma}, h)}\pdbap$, with ${ \overline v}$   a vector field along the path ${ \overline\gamma}:[t_0,t_1]\to P$ belong to  $\pap$ and $X\in T_hH$.

At this stage we observe that the connection $\Omega$ on the  principal $H\rtimes_{\alpha}G$-bundle $\pdbap\to\mpm$ is analogous to a connection $\ovA$ on the principal $G$-bundle $P\to M$. {\em Just as we decorated the $\ovA$-horizontal paths on $P$ with elements of a new group $H$, we will decorate the $\Omega$-horizontal paths on $\pdbap$ with elements from a new group $K$.}

Let us then examine the space of  $\Omega$-horizontal paths. This is given by
\begin{equation}\label{E:pdec1OmA}
{\mathcal P}_{\Omega}(\pdbap),
\end{equation}
whose points are $\Omega$-horizontal paths
$$\ovG h:[s_0,s_1]\to  \pdbap :s\mapsto (\ovG_s,h_s).$$
 As with paths on $P$ and on $M$ we impose the technical conditions discussed in    the context of (\ref{E:Gamma}). We have also projection mapping
 \begin{equation}\label{E:dec1OmAbun}
\pi_{\Omega}: {\mathcal P}_{\Omega}(\pdbap) \to\mpm:\ovG h\mapsto \pi_{\ovA}(\ovG)=\pi\circ\ovG.
 \end{equation}
 The right action of $H\rtimes_{\alpha}G$ on $\pdbap$ maps   $\Omega$-horizontal paths on $\pdbap$ to $\Omega$-horizontal paths.  In this way ${\mathcal P}_{\Omega}(\pdbap)$ is a principal $H\times_{\alpha}G$-bundle over $\mpm$.

With the objective of moving to a second decorating structure group $K$, let us consider a second Lie crossed module
\begin{equation}\label{E:K2G2}
(H\rtimes_{\alpha}G, K, \alpha_2, \tau_2).
\end{equation}
Decorating an element $\ovG h\in\pdbap$ with an element of $K$ produces a pair
\begin{equation}\label{E:Gamhk}
 (\ovG h, k),
\end{equation}
 where 
 $$\ovG h: [s_0,s_1]\to P\times H: s\mapsto ({ \overline\Gamma}_s,h_s)$$ 
 is a path belonging to $\pdbap$ that is $\Omega$-horizontal and $k\in K$. (Since $s\mapsto h_s$ is determined, through the differential equation (\ref{E:dothdot}),  by its initial value $h_{s_0}$ we could simply use this initial value in (\ref{E:Gamhk})). These pairs yield the  {\em doubly decorated} bundle  
  \begin{equation}\label{E:decdecbun}
\pi_{2}^{\rm dec}: {\mathcal P}_{\Omega}^{\rm dec}(\pdbap)\stackrel{\rm def}{=}  {\mathcal P}_{\Omega}(\pdbap)\times K \to\mpm:(\ovG h, k)\mapsto \pi_{\ovA}(\ovG)=\pi\circ\ovG.
 \end{equation}
 The new structure group is    $K\rtimes_{\alpha_2}(H\rtimes_{\alpha}G)$.
 
 The procedure of passing to decorated structures and connections can be expressed very efficiently and comprehensively using a category-theoretic framework which we have developed in \cite{CLS2geom}. We shall now describe this framework; further details and proofs are available in \cite{CLS2geom}. 
 
   \subsection{Categorical groups}
 
 A {\em categorical group} ${\mathbf G}$ is a category along with a functor
 $$\mbg\times\mbg\to\mbg$$
 that makes both $\Obj(\mbg)$ and $\Mor(\mbg)$ groups.  The functoriality says, explicitly, that
 \begin{equation}\label{E:exchangelaw}
 (f'_2\circ f'_1)(f_2\circ f_1)=(f'_2f_2)\circ (f'_1f_1)
 \end{equation}
 for all morphisms $f_i, f'_j$ for which the composites on the left are defined. We note two useful consequences: (i) the source and target maps
 $$s, t:\Mor(\mbg)\to\Obj(\mbg)$$
 are homomorphisms, and (ii) $\Obj(\mbg)\to\Mor(\mbg):a\mapsto 1_a$  is a homomorphism.

 For most cases of interest to us, both $\Obj(\mbg)$ and $\Mor(\mbg)$ are Lie groups and the source and target maps are smooth. Usually we denote the object group $\Obj(\mbg)$ by $G$, and we denote by $H$ the kernel of the source map:
 $$H=\ker s\subset \Mor(\mbg).$$
 The restriction of  the target map $t$ to $H$ gives a homomorphism
 $$\tau:H\to G.$$
 There is an action of $G$ on $H$, given by
\begin{equation}\label{E:hactG}
h\mapsto \alpha(g)(h)=1_gh1_g^{-1},
\end{equation}
 where $1_g:g\to g$ is the usual identity morphism at the object $g\in G$; thus $\alpha:G\to {\rm Aut}(H)$ is a homomorphism. It can be checked then that the map
 $$\Mor(\mbg)\to H\rtimes_{\alpha}G: \phi\mapsto\Bigl(\phi1_{s(\phi)}^{-1}, s(\phi)\Bigr)$$
 is an isomorphism of groups, thereby identifying $\Mor(\mbg)$ with the semidirect product group $H\rtimes_{\alpha}G$.  Under this identification, the source and target maps become:
 \begin{equation}\label{E:semidst }
 s(h,g)=g\qquad\hbox{and}\qquad t(h,g)=\tau(h)g.
 \end{equation}
 The composition of morphisms correspond to the following operation on $H\rtimes_{\alpha}G$:
 \begin{equation}\label{E:semidcompo}
 (h_2,g_2)\circ (h_1,g_1)=(h_2h_1,g_1),
 \end{equation}
 for all $(h_1,g_1), (h_2,g_2)\in H\times G$ for which $\tau(h_1)g_1=g_2$.
 
 The system  $(G,H,\alpha,\tau)$ is called a {\em crossed module}, and in the case where $G$ and $H$ are Lie groups and  $\tau$ as well as the action map $(g,h)\mapsto\alpha(g)(h)$ are smooth,  it is called a {\em Lie crossed module}. 
 
 We refer to Kelly and Street \cite{KS} for   more details and references on the subject of crossed modules, 2-groups and related notions.

   \subsection{Categorical principal bundles}
 
We will work with a categorical group ${\mathbf G}$, with associated Lie crossed module $(G,H,\alpha,\tau)$. 
By a {\em right action} of ${\mathbf G}$ on a category ${\mathbf P}$ we mean a functor
\begin{equation}\label{D:rightact}
\mbp\times\mbg\to\mbp
\end{equation}
that is a right action both at the level of objects and at the level of morphisms.  

By a {\em principal categorical bundle}  we mean a functor $\pi:\mbp\to\mbbb$ along with a right action of a categorical group $\mbg$ on $\mbp$ satisfying the following conditions:
\begin{itemize}
\item[(b1)] $\pi$ is surjective both at the level of objects and at the level of morphisms;
\item[(b2)] the action of $\mbg$ on $\mbp$ is free on both objects and morphisms;
\item[(b3)]  the action of $\Obj(\mbg)$ preserves the fiber $\pi^{-1}(b)$ and its action on this fiber is transitive, for each object $b\in\Obj(\mbbb)$; the action of $\Mor(\mbg)$ preserves the fiber   $\pi^{-1}(\phi)$ and its action on this fiber is transitive,  for each morphism  $\phi\in\Mor(\mbbb)$.
\end{itemize}
We think of the objects of $\mbbb$ as points of a base   space $B$, and a morphism $\gamma\in\Mor(\mbbb)$  as a path in $B$ running from the `initial point' given by the source $s(\gamma)$ to the `final point' given by the target $t(\gamma)$;  we think of $\mbp$ similarly, although other useful interpretations are also possible. 
 We are not imposing any form of local triviality in this context.
  
 Let us pause to look at two examples briefly (without technical details). The first is simply to show that the definition includes that of traditional principal bundles as a special case. For this consider a principal $G$-bundle $\pi:P\to B$, and let $\mbbb$ be the discrete category with object set $B$ and $\mbp$ the discrete category with object set $P$; this forms a categorical principal bundle in the obvious way with the structure categorical group being discrete with  object group $G$. 
 
 Much more relevant to our point of view  is the following example. We start with a connection $\ovA$ on a principal $G$-bundle $\pi:P\to B$; for $\mbbb$ we take as objects the points of $B$ and as morphisms the paths (of suitably smoothness and reparamerization equivalence class so as to make composition of morphisms meaningful) in $B$, whereas for $\mbp$ we take $P$ as object set but for morphisms we take $\ovA$-horizontal paths. This still gives a categorical principal bundle whose structure group is the discrete categorical group with object group $G$. 
 
 To obtain a richer class of examples, we consider the {\em decorated bundle }
\begin{equation}\label{E:decPA}
{\catdeca},
\end{equation}
  whose objects are the points of $P$ and whose morphisms are of the form
 $$({ \overline\gamma}, h)$$
 where ${ \overline\gamma}$ is any $\ovA$-horizontal path on $P$ and $h\in H$, where $(G,H,\alpha,\tau)$ is the Lie crossed module corresponding to the categorical group $G$. The source and target maps are 
 \begin{equation}
\label{E:stdec}
s({ \overline\gamma}, h)=s({ \overline\gamma})\qquad\hbox{and}\qquad t({ \overline\gamma}, h)= t({ \overline\gamma})\tau(h).
\end{equation}
 The law of composition is
 \begin{equation}\label{E:compmor}
({ \overline\gamma}_2,h_2 )\circ ( { \overline\gamma}_1,h_1)=({ \overline\gamma}_3, h_1h_2),
\end{equation}
where ${ \overline\gamma}_3$ is the composite  of the path ${ \overline\gamma}_1$ with the right translate ${ \overline\gamma_2}\tau(h_1)^{-1}$ (so that the final and initial points match correctly):
\begin{equation}\label{E:tilgamcomp}
{ \overline\gamma}_3={ \overline\gamma_2}\tau(h_1)^{-1}\circ { \overline\gamma}_1. \end{equation}
The reversed ordering $h_1h_2$ on the right side in (\ref{E:compmor}) is necessary and ensures proper behavior of the target map:
\begin{equation}\label{E:targetcheck}
t({ \overline\gamma}_3,h_1h_2)=t({ \overline\gamma}_2)\tau(h_1)^{-1} \tau(h_1h_2) =t({ \overline\gamma}_2,h_2),
\end{equation}
and the corresponding result for the sources is clear. There is a right action of $\mbg$ on ${\catdeca}$: on objects it is just the right action of $G$ on $P$, and on morphisms it is given, as we have discussed in (\ref{gract}), by
\Beq\label{gract2}
(\ovg, h)(h_1, g_1):=(\ovg g_1, \alpha(g_1^{-1})(h h_1)).
\Eeq
With this action, $\pi_{\ovA}:{\catdeca}\to\mbbb$ is a principal categorical bundle.
 
   \subsection{Categorical connections}
   
We turn now to a categorical framework for geometry on categorical bundles. By a {\em categorical connection} $\mbba$ we mean an assignment to each $\gamma\in \Mor(\mbbb)$ and each object $u\in\pi^{-1}\bigl(s(\gamma)\bigr)$ of a morphsim ${ \overline\gamma}_u\in\Mor(\mbp)$, {\em the $\mbba$-horizontal lift initiating at} $u$, satisfying the following conditions: 

 \begin{itemize}
 \item[(h1)] $s({ \overline\gamma}_u)=u$
 \item[(h2)] the morphism ${ \overline\gamma}_u$ projects down to $\gamma$: 
$$\pi\bigl({ \overline\gamma}_u\bigr)=\gamma;$$
\item[(h3)] the lifting   is functorial: if $\zeta\in\Mor(\mbbb)$ is   a morphism with source $s(\gamma)$ then the $\mbba$-horizontal lift of ${\zeta}\circ{\gamma}$ through $u$ is
\begin{equation}\label{E:horliftfunct}
{ \overline \zeta}^{\rm hor}_v\circ { \overline \gamma}^{\rm hor}_u,\end{equation}
where $v=t\bigl( { \overline \gamma}^{\rm hor}_u\bigr)$;
\item[(h4)]  a `rigid vertical motion' of a horizontal morphism  produces a horizontal morphism:  
\begin{equation}\label{E:horliftrightact}
{ \overline \gamma}^{\rm hor}_{u}1_g
={ \overline \gamma}^{\rm hor}_{ug} \end{equation}
for all $g\in \Obj(\mbg)$. Identifying morphisms of $\mbg$ with pairs  $(h,g)\in H\rtimes_{\alpha}G$, this condition reads
\begin{equation}\label{E:horliftrightact2}
 \gamma^{\rm hor}_{u}\cdot (e,g) = \gamma^{\rm hor}_{ug}.
\end{equation}
\end{itemize}
 
Consider now  the categorical bundle 
$$\pi_{\ovA}:\mbp_{\ovA}^{\rm dec}\to\mbbb,$$
that we described in the context of (\ref{E:decPA}). 
Let us    construct a categorical connection on this bundle. For this purpose we use an  $L(H)$-valued $\alpha$-equivariant $1$-form $C_0$ on $P$ that vanishes on vertical vectors.   This $1$-form allows us to obtain a decorating element $h^*({ \overline\gamma}_u)$ for each $\ovA$-horizontal path ${ \overline\gamma}_u:[t_0,t_1]\to P$   as follows:  $h^*({ \overline\gamma}_u)$ is the value at $t=t_1$ of the solution of the differential equation
\begin{equation}\label{E:dhuC}
 h'(t)h(t)^{-1} =-C_0\bigl({ \overline\gamma}_u'(t)\bigr)\qquad t\in [t_0,t_1],
\end{equation}
with initial condition $h(t_0)=e\in H$. With this in place we can state our specification of a categorical connection on the bundle $\pi_{\ovA}:\mbp_{\ovA}^{\rm dec}\to\mbbb$; we take the horizontal lift ${ \overline \gamma}^{\rm hor}_u$ of any $\gamma\in \Mor(\mbbb)$, with any source $u\in\pi_{\ovA}^{-1}\bigl(s(\gamma)\bigr)$, to be given by
\begin{equation}\label{E:tildgamcatcon}
{ \overline \gamma}^{\rm hor}_u=\bigl({ \overline\gamma}_u, h^*({ \overline\gamma}_u)\bigr)
\end{equation}
where ${ \overline\gamma}_u\in\pap$ is the $\ovA$-horizontal lift of $\gamma$ initiating at $u$.

The decorating element $h^*({ \overline\gamma}_u)$ can be understood as the `difference' between parallel transport by the connection $\ovA$ and by the connection $\ovA+\tau {C_1}$, as we now explain.

\begin{prop}\label{P:horliftbarAC}  Suppose $(G,H,\alpha,\tau)$ is a Lie crossed module, and  $\ovA$ and ${\hat A}$  are  connection forms on a principal $G$-bundle $\pi:P\to M$, such that
$${\hat A}=\ovA+\tau {C_1},$$
where ${C_1}$ is an $\alpha$-equivariant $L(H)$-valued $1$-form on $P$ that vanishes on vertical vectors.  Let   ${ \hat\gamma}_u$ be the ${\hat A}$-horizontal lift, with inital point $u\in P$,  of a path $\gamma$ on $B$, initiating at $u$, and ${\overline\gamma}_u$ the $\ovA$-horizontal lift, also with initial point $u$, of $\gamma$.
Then  the terminal point of ${ \hat\gamma}_u$ is
\begin{equation}\label{E:tgamhat}
t\bigl({ \hat\gamma}_u\bigr)=t({ \overline\gamma}_u)\tau\bigl(h^*({ \overline\gamma}_u)\bigr),
\end{equation}
where the element $h^*({ \overline\gamma}_u)\in H$ is the final point of the path $[t_0,t_1]\to H: t\mapsto h(t)$  satisfying the differential equation \begin{equation}\label{E:dhuC2}
\begin{split}
 h'(t)h(t)^{-1}& =-{C_1}\bigl({ \overline\gamma}_u'(t)\bigr)\qquad \hbox{for all $t\in [t_0,t_1]$,}\\
 h(t_0) &=e.
 \end{split}
\end{equation}
\end{prop}
\begin{proof} Let $h(t)\in H$ be as described above in the context of (\ref{E:dhuC}). Consider now the path
\begin{equation}\label{E:tauht}
[t_0,t_1]\to P:t\mapsto { \overline\gamma}_u(t)\tau\bigl(h(t)\bigr).
\end{equation}
Before proceeding further to the main computation, let us observe that $\tau {C_1}$ is an $L(G)$-valued Ad-equivariant $1$-form on $P$ that vanishes on vertical vectors. In particular, applying $\tau {C_1}$ to a vector of the form
$$v\tau\bigl(h(t)\bigr),$$
where $v\in T_{p}P$ with $p={ \overline\gamma}_u(t)$, produces
$${\rm Ad}\bigl(\tau\bigl(h(t)\bigr)\bigr)^{-1} \tau {C_1}(v).$$

Applying the $1$-form $\ovA+\tau {C_1}$ to the derivative of the path (\ref{E:tauht}) we have
\begin{equation}
\begin{split}
(\ovA+\tau {C_1})\Bigl({\overline\gamma}_u'(t)\tau\bigl(h(t)\bigr) +{ \overline\gamma}_u(t)\tau\bigl(h'(t)\bigr)\Bigr) &\\
&\hskip -2in =0+ \tau {C_1}\Bigl({ \overline\gamma}_u'(t)\tau\bigl(h(t)\bigr)\Bigr) \\
&\hskip -1.25in +{\bar A}\Bigl( { \overline\gamma}_u(t)\tau\bigl(h'(t)\bigr)\Bigr)
+ \tau {C_1}\Bigl({ \overline\gamma}_u(t)\tau\bigl(h'(t)\bigr)\Bigr)\\
&\hskip -2in =0+ {\rm Ad}\bigl(\tau\bigl(h(t)\bigr)\bigr)^{-1} \tau {C_1}\Bigl({ \overline\gamma}_u'(t)\Bigr)+\tau\bigl(h(t)^{-1}h'(t)\bigr) +0\\
&\hskip -2in =\tau\bigl(h(t)^{-1}\bigr)\tau\Bigl[{C_1}\Bigl({ \overline\gamma}_u'(t)\Bigr) + h'(t)h(t)^{-1}\Bigr]\tau\bigl(h(t)\bigr)\\
&\hskip -2in =0,
\end{split}
\end{equation}
upon using the equation (\ref{E:dhuC}) satisfied by $h(\cdot)$. Thus the path (\ref{E:tauht}) is horizontal with respect to the connection for ${\bar A}+\tau {C_1}$. Its terminal point is
$${\overline\gamma}_u(t_1)h(t_1)=t({ \overline\gamma}_u)\tau\bigl(h^*({ \overline\gamma}_u)\bigr),$$
 by definition of $h^*({ \overline\gamma}_u)$. \end{proof}
 
 The significance of the preceding result is that the decorating element $h^*({ \overline\gamma}_u)$ arises from parallel transport of $u$ by a second connection form
 $\ovA+\tau {C_1}$. A larger point is that   {\em a categorical connection on $\catdeca\to\mbbb $ corresponds to a differential geometric connection on the bundle } $P\to M$. We can state this more formally.
 
 \begin{prop}\label{P:catdgconn} Let   $\ovA$ be  a  connection form  on a principal $G$-bundle $\pi:P\to M$, and $\pap\to\mpm$ the  principal $G$-bundle of $\ovA$-horizontal paths $[t_0,t_1]\to P$ and $\mpm$ the space of paths $[t_0,t_1]\to M$.  As usual, paths that differ by time-translation are identified and the paths are $C^\infty$, constant near initial and final times $t_0$ and $t_1$. Let $\mbg$ be a categorical group with associated Lie crossed module $(G, H,\alpha,\tau)$.  Now suppose $\ovA$ is a connection form on $P$ and ${C_1}$ is an $L(H)$-valued $1$-form on $P$ that vanishes on vertical vectors and is $\alpha$-equivariant as in Proposition \ref{P:horliftbarAC}; thus $A=\ovA+\tau {C_1}$ is a connection form on $P$. For $\ovg\in\pap$ let $h^*(\ovg)\in H$ be also as in Proposition \ref{P:horliftbarAC}.  Now consider the principal categorical bundle $\catdeca\to \mbbb$, with structure categorical group $\mbg$, described above in the context of (\ref{E:dhuC}). Then a categorical connection on this bundle is provided by the following specification: for any path $\gamma\in \Mor(\mbbb)$ and object $u\in \Obj(\catdeca)$ lying in the fiber over $\gamma(t_0)$ we take as horizontal lift the morphism $\bigl(\ovg_u, h^*(\ovg)\bigr)\in\Mor(\catdeca)$, where $\ovg_u$ is the $\ovA$-horizontal lift of $\gamma$ with initial point $u$. \end{prop}

 \subsection{Categorical connections for  the doubly decorated bundle} We proceed now to a higher version of Proposition \ref{P:catdgconn}. The decorated category  $\catdeca$ has as objects the points $p\in P$ and as morphisms the pairs $(\ovg, h)\in\pap\times H$.  We consider next a doubly decorated categorical bundle 
 $$\mbp_2\to \mbbb_2.$$
 The structure group is a second categorical group $\mbg_2$, with corresponding Lie crossed module being $(H\rtimes_{\alpha}G, K, \alpha_2,\tau_2)$. 
 
The objects of $\mbbb_2$ are the morphisms $\gamma$ of $\mbbb$ (that is, paths on $M$):
\begin{equation} \Obj(\mbbb_2) =\Mor(\mbbb),\end{equation} 
 and the morphisms are paths $\Gamma:s\mapsto\Gamma_s\in\mpm$ of paths on $M$.
 
The objects of $\mbp_2$ are the morphisms $(\ovg, h)$ of $\pdbap$:
\begin{equation} \Obj(\mbp_2) =\Mor(\mbp).\end{equation} 
  Morphisms of  $\mbp_2$ are of the form $\bigl(\ovG  h, k\bigr)$ where 
$$\ovG h:[s_0,s_1]\to \pdbap: s\mapsto (\ovG_s, h_{s})$$
is a path of  decorated paths on $P$,  horizontal with respect to a connection  such as $\Omega$, and $k\in K$ is a `higher decoration.' 

A categorical connection on $\mbp_2\to \mbbb_2$ picks out a special decoration 
$$k_{s_1} \in K $$
for each  horizontal decorated path $\ovG h$.  Applying $\tau_2$ to this special element of $K$ yields an element of $H\rtimes_{\alpha}G$. The object
\begin{equation}
(\ovg_{s_1}, h_{s_1})\tau_2(k_{s_1})\in\pdbap
\end{equation} 
is the terminal point of  $\bigl(\ovG h, k_{s_1}\bigr)$:
\begin{equation}\label{E:termpt}
t\Bigl(\bigl(\ovG h, k_{s_1}\bigr)\Bigr)=(\ovG_{s_1}, h_{s_1})\tau_2(k_{s_1}).
\end{equation}
In more detail, from the categorical connection point of view, the horizontal lift of the path of paths
$$[s_0,s_1]\to \mpm:s\mapsto\Gamma_s,$$
viewed as a morphism of $\mbbb_2$, with initial point $(\ovG_{s_0}, h_{s_0})$, is   of the form
\begin{equation}
  \bigl(\ovG h, k_{s_1}\bigr),
\end{equation}
where
$$\ovG h:s\mapsto (\ovG_s, h_{s})$$
is horizontal with respect to a fixed given connection on the bundle
$$\pdbap\to\mpm,$$
and $k_{s_1}$ is a special element of $K$ associated to the path $\ovG h$.
The target of this horizontal lift morphism is 
\begin{equation}\label{E:termpt2}
(\ovG_{s_1}, h_{s_1})\tau_2(k_{s_1}).
\end{equation}
If the categorical connection is such that $\tau_2(k_{s_1})$ lies in the subgroup $H\subset H\rtimes_{\alpha}G$ then the target is
\begin{equation}\label{E:termpt3}
\bigl(\ovG_{s_1}, h_{s_1} \tau_2(k_{s_1})\bigr).
\end{equation}
  
 We focus on a type of categorical connection on $\mbp_2\to\mbbb_2$ that arises from  a differential geometric connection $\Omega$ on the principal $ (H\rtimes_{\alpha}G)$-bundle 
 \begin{equation}\label{E:decdecbun2}
\pi^{\rm dec}:  \pdbap  \to\mpm: \ovG h \mapsto \pi_{\ovA}(\ovG)=\pi\circ\ovG 
 \end{equation}
 with an added `horizontal' term. We are constructing an analog of the connection form $A+\tau {C_1}$ described in Proposition \ref{P:catdgconn}.
 Instead of the $L(H)$-valued $1$-form ${C_1}$ on $P$ we now use an $L(K)$-valued $1$-form on $\pdbap$ given by
 \begin{equation}\label{E:Cdec}
 \begin{split}
 C^{\rm dec}_{(\ovg, h)}(\ovv+X) &\\
 &\hskip -.5in ={\rm Ad}(h^{-1})\left[C^R_2|_{\ovg(t_2)}\bigl(\ovv(t_1)\bigr)-C^L_2|_{\ovg(t_1)}\bigl(\ovv(t_0)\bigr)  +\int_{t_0}^{t_1}{D}\bigl(\ovv(t), \ovg'(t)\bigr)\,dt\right],
 \end{split}
 \end{equation}
 where we have kept in mind the expression for $\Omega$ given in (\ref{deccon})  and $C^{L,R}_2$ are $L(K)$-valued $1$-forms, vanishing on vertical vectors and  $\alpha_2$-equivariant., and ${D}$ is an $L(K)$-valued $2$-form that vanishes when contracted with any vertical vector and is also $\alpha_2$-equivariant. This equivariance property, spelt out for ${D}$, is
 \begin{equation}\label{E:Eequiv}
 {D}_{pg}(v,w)=\alpha_2(g^{-1}){D}_p(v,w)
 \end{equation}
 for all $p\in P$,  vectors $v, w\in T_pP$, and in the notation $\alpha_2(g)$
 we have identified $G$ as a subgroup of $H\rtimes_{\alpha}G$.
 
 Then, analogously to Proposition \ref{pr:decconnect2} (from where we also use the notation), we have
 \begin{equation}\label{E:Cdecequiv}
 C^{\rm dec}_{(\ovg,h)h_1g_1}{\mathcal R}_{(h_1,g_1)}|_{*}(\ovv,X)={\rm Ad}(h_1g_1)^{-1} C^{\rm dec}_{(\ovg,h)}(\ovv,X)
 \end{equation}
 for all $h, h_1\in H$, $g_1\in G$, $\ovg\in\pap$, $\ovv\in T_{\ovg}\pap$ and $X\in T_hH$.  Moreover,
 \begin{equation}\label{E:Cdecvert}
  C^{\rm dec}_{(\ovg,h)}r_{(\ovg,h)}(Y_1+Z_1) =0
 \end{equation}
 for all $(Y_1, Z_1)\in L(H)\oplus L(G)$.  We omit the proof which is a straightforward analog of that of Proposition \ref{pr:decconnect2}.

  As consequence we obtain a connection form
 \begin{equation}\label{E:Omdec}
 \Omega^{\rm dec}= \Omega+\tau_2 C^{\rm dec}
 \end{equation}
 on the decorated bundle $\pi^{\rm dec}:  \pdbap  \to\mpm$.
 
 The form $ C^{\rm dec}$ then provides a categorical connection $\mbba_2$ on the categorical bundle
 $$\mbp_2\to\mbbb_2.$$
 To a morphism $\ovG h\in\Mor(\mbp_2)$ it associates an element
 $$k^*(\ovG h)\in K,$$
 which is the terminal value $k_{s_1}$ of the solution for  $K$-valued path $[s_0,s_1]\to K:s\mapsto k_s$ that satisifies the differential equation
 \begin{equation}
 {\dot k}_sk_s^{-1}=- C^{\rm dec}_{(\ovG_s,h_s))}({\dot\ovG}_s, {\dot h}_s)
 \end{equation}
 subject to $k_{s_0}=e$.  Here, as usual, a dot over a letter denotes the derivative  $\partial_s$ with respect to $s$.

\section{Curvature conditions for reduction to  holonomy bundle}\label{s:reduct}

Consider a general connection  $\Omega$ on a decorated bundle $ {\mathcal P}_{\bar A}^{\rm dec}P$.  Let $ {C_1}$
 be an $L(H)$-valued $1$-form on $P$ that is equivariant and vanishes on vertical vectors.  Then we can associate to each $\ovg\in\pap$ a special decoration $h^*({{\ovg}})$ 
 that is given by
\begin{equation}\label{E:hstargam}
h^*({{\ovg}})=h_{\ovg}(t_1),
\end{equation}
where $ [t_0,t_1]\to H: t\mapsto {  h}_{\tlg}(t)$ is the solution of
$${  h}'_{\ovg}(t){  h}_{\ovg}(t)^{-1}=-  {C_1}\bigl({\ovg}'(t)\bigr),$$
with initial value ${  h}_{\ovg}(t_0)=e$.

Then we have a sub-bundle 
$ {\bar {\mathcal P}}_{\bar A}^{\rm dec}P$  of $ {\mathcal P}_{\bar A}^{\rm dec}P$ specified by:
\begin{equation}\label{E:subbundle}
{\bar  {\mathcal P}}_{\bar A}^{\rm dec}P:=\{\bigl({\ovg}, h^*({{\ovg}})^{-1}\bigr)\,| \, \ovg\in\pap\}\subset {\mathcal P}_{\bar A}^{\rm dec}P.
\end{equation}
More precisely, 
\begin{equation}\label{E:subbun}
{\bar  {\mathcal P}}_{\bar A}^{\rm dec}P \to \mpm: \bigl({\ovg}, h^*({{\ovg}})^{-1}\bigr)\mapsto \pi\circ\ovg
\end{equation} 
is a principal $G$-bundle.

Our goal in this section is to determine a type of connection   $\hat\Omega$ on $ {\mathcal P}_{\bar A}^{\rm dec}P$ that reduces to a connection on the sub-bundle ${\bar  {\mathcal P}}_{\bar A}^{\rm dec}P$. Subsection \ref{ss:vardiff} serves as a technical appendix to this section and presents some of the background computations for   the proof of the main result Proposition \ref{P:reduction}.  
Later in subsection \ref{ss:holb} we present  a description of the notion of a holonomy bundle and a result of Ambrose and Singer  \cite{AmSi1953}, that form  a motivational background for   our investigations.

\subsection{Statement of the result}\label{ss:main}

Let us first summarize  the notation and framework.

Let $\ovA$  and $A$ be  connection forms on a principal $G$-bundle $\pi:P\to M$, and suppose $(G, H, \alpha,\tau)$ and $(H\rtimes_{\alpha}G, K,\alpha_2,\tau_2)$ are Lie crossed modules. (As noted before, we will use only the semidirect products here, not  the target maps $\tau$ and $\tau_2$). Let ${C}$ be an $L(H\rtimes_{\alpha}G)$-valued $1$-form on $P$ that vanishes on vertical vectors and satisfies the equivariance
 \begin{equation}\label{E:C0equiv}
 {C}|_{pg}(vg)={\rm Ad}(g^{-1}){C}|_p(v)
 \end{equation}
 for all $p\in P$, $v\in T_pP$ and $g\in H$. Here, as usual, on the right  we take ${\rm Ad}(g^{-1})$ to be an operator on $L(H\rtimes_{\alpha}G)$. We decompose $C$ into its component in $L(H)$ and the component in $L(G)$:
 \begin{equation}
 C=C_0+C_1,
 \end{equation}
 where $C_0$ takes values in $L(G)$ and $C_1$ in $L(H)$.

 Let ${B} $ be an $L(H\rtimes_{\alpha}G)$-valued $1$-form on $P$ that is equivariant analogously to $C$ and vanishes  when contracted on any vertical vector; we write
 \begin{equation}\label{E:BB01}
 B=B_0+B_1,
 \end{equation}
 where $B_0$ is the $L(G)$-component of $B$ and $B_1$ is the $L(H)$-component. 	`
 	
  Let $\pap$ be, as before, the path space of  $\ovA$-horizontal paths $\ovg$ on $P$, and $\mpm$ the path space for $M$.  On the  decorated bundle  
  $$\pdbap\stackrel{\rm def}{=}\pap\times H\to \mpm$$
  consider the connection form $\hat\Omega$ given by
 \begin{equation}\label{E:hatOmfull2}
\begin{split}
  {\hat\Omega}_{{\ovg}, h}({\ovv}+X)  
&={\rm Ad}(h^{-1})\left[{\ovA}\bigl(\ovv(t_0)\bigr) + {\rm Ad}\bigl(g^*({\ovg})^{-1} \bigr){C}|_{\ovg(t_1)}\bigl(\ovv(t_1)\bigr)\right. \\
&\hskip 1in \left. -{C}|_{\ovg(t_0)}\bigl(\ovv(t_0)\bigr)\right. \\
&\left.\hskip  1in+\int_{t_0}^{t_1}{\rm Ad}\bigl(g_{\ovg(t)}^{-1}\bigr) B_{\ovg(t)}({\ovv}(t), {\ovg}'(t)\bigr) \,dt +Xh^{-1}\right]
\end{split}
\end{equation} 
for all $\ovg\in\pap$, $h\in H$, $\ovv\in T_{\ovg}\pap$ and $X\in T_hH$.
 Suppose that   $B_1$ and ${C_1}$ are related by
 \begin{equation}\label{E:BDC0}
 B_1 + d{{  C}_1}+\frac{1}{2}[{{  C}_1},{{  C}_1}] =0.
 \end{equation}
 Consider a path $[s_0,s_1]\to\pap: s\mapsto { \Gamma}_s$ given by a $C^\infty$ map
 $$[t_0,t_1]\times[s_0,s_1]\to M:(t,s)\mapsto\Gamma_s(t)= \Gamma(t,s)$$
 and let $\ovg\in\pap$ be an $\ovA$-horizontal lift of the initial path $\Gamma_{s_0}$. Let $h^*(\ovg)$ be the value $h_{s_0}(t_1)$ where  $[t_0,t_1]\to H:t\mapsto h_{s}(t)$ is the solution of the equation
 \begin{equation}
 \begin{split}
 h_{s}'(t)+{C_1}\bigl(\ovg'(t)\bigr)h_{s_0}(t)&=0\quad\hbox{ for all $t\in [t_0,t_1]$; }\\
 h_{s_0}(t_0)&=e.
 \end{split}
  \end{equation}
  Let
  \begin{equation}\label{E:hstargstar2}
\begin{split}
h^*(\ovg) &= h_{\ovg}(t_1)\\
g^*(\ovg) &=\tau\bigl(h^*(\ovg)\bigr).
\end{split}
\end{equation}
Then the parallel transport by $\hat\Omega$ of an element $(\ovg,h^*(\ovg)^{-1})$ along the path 
$$[s_0,s_1]\to\pap: s\mapsto { \Gamma}_s$$
 given by a $C^\infty$ map
 $$[t_0,t_1]\times[s_0,s_1]\to M:(t,s)\mapsto { \Gamma}_s(t)$$
is given by
\begin{equation}
s\mapsto \bigl({\hat\Gamma}_s, h_s^*(\tlG_s)\Bigr)
\end{equation}
where $s\mapsto{\hat\Gamma}_s$ is the horizontal lift of $s\mapsto\Gamma_s$ by the connection form $\hat\omega$ on $\pap$ that is given by 
  \begin{equation}\label{E:hatOmhat3}
\begin{split}
& {\hat\omega}_{{\ovg}}({\ovv}) \\
&={\rm Ad}(h^{-1})\left[{\ovA}\bigl(\ovv(t_0)\bigr) +{\rm Ad}\bigl(g^*({\ovg})^{-1} \bigr){C}|_{\ovg(t_1)}\bigl(\ovv(t_1)\bigr)\right. \\
&\hskip 1in \left. -{C}|_{\ovg(t_0)}\bigl(\ovv(t_0)\bigr)\right. \\
&\left.\hskip  1in+\int_{t_0}^{t_1}{\rm Ad}\bigl(g_{\ovg(t)}^{-1}\bigr)   B_{\ovg(t)}({\ovv}(t), {\ovg}'(t) \,dt \right]
\end{split}
\end{equation}

 \begin{prop}\label{P:reduction} With notation and framework as above, suppose the relations (\ref{E:C0C1}) hold. Then parallel transport by the connection $\hat\Omega$ carries elements of the form $\bigl(\ovg, h^*(\ovg)^{-1}\bigr)$ to elements of the same form.  \end{prop}

 The remainder of this section is devoted to an understanding and proof of this result.

\subsection{A pullback connection} Let us consider on the bundle $\pi:P\to M$    a second connection form
\begin{equation}\label{E:Abarprim}
\ovA'=\ovA+ \tau C_1.
\end{equation}
Corresponding to this there is a bundle of horizontal paths
\begin{equation}
\papp\to \mpm,
\end{equation}
and a   decorated bundle
$$\pi^{\rm dec}:\papp\times H\to\mpm.$$
 On this decorated bundle we have the connection form $\Omega$ given by
\Beq\label{E:flatOmega}
\begin{split}
&  \Omega_{{\tlg}, h}({\tlv}+X) \\
&={\rm Ad}(h^{-1})\left[{\ovA}\bigl(\ovv(t_0)\bigr) + C^R|_{\tlg(t_1)}({\tlv}(t_1))- C^L|_{\tlg(t_0)}({\tlv}(t_0)) \right. \\
&\left. \hskip 1in +   \int_{t_0}^{t_1}{B}_{\tlg(t)}\bigl({\tlv}(t), {\tlg}'(t)\bigr)\,dt+Xh^{-1}\right],
\end{split}
\Eeq
with ${B_1}$ being an   $L(H)$-valued $2$-form on $P$ that vanishes on vertical vectors and has the usual equivariance property.   

Recall the bijection
\begin{equation}\label{E:mct2}
\mct:\pap\to\papp:\ovg\mapsto \mct(\ovg)
\end{equation}
where $\mct(\ovg)$ is the $\ovA'$-horizontal path on $P$ that has the same initial point $\ovg(t_0)$ as $\ovg$ and projects down to $\pi\circ\ovg$.  At the level of the decorated bundles we define the map
\begin{equation}\label{E:mctdec}
\mct_d:\pap\times H\to\papp \times H:(\ovg, h) \mapsto \bigl(\mct(\ovg), h^*({\ovg})^{-1}h\bigr)
\end{equation}
where $h^*(\ovg)$ is as in (\ref{E:hstargam}).
The motivation for this definition arises from the target formula (\ref{E:tgamhat}) which says that
$$t\bigl(\mct(\ovg)\bigr)=t({ \overline\gamma})\tau\bigl(h^*({ \overline\gamma})\bigr);$$
this implies
\begin{equation}\label{E:mctdectarget}
\begin{split}
t\bigl(\mct(\ovg), h^*({\ovg})^{-1}h\bigr) &=t\bigl(\mct(\ovg)\bigr)\tau\bigl( h^*({\ovg})^{-1}\bigr)\tau(h)\\
& =t(\ovg)\tau(h)\\
&=t(\ovg, h).
\end{split}
\end{equation}
In Proposition \ref{P:isomor} we have worked out a `derivative' map
$$\mct_*:T_{\ovg}\pap\to T_{\mct(\ovg)}\papp. $$
With this in mind it is natural to take the derivative of $\mct_d$ to be given by
\begin{equation}\label{E:mctstardec}
\begin{split}
(\mct_d)_*:T_{(\ovg, h) }\Bigl(\pap\times H\Bigr) &\to T_{\mct_d(\ovg, h) }\Bigl(\papp \times H\Bigr)\\
& (\ovv, hX) \mapsto   \Bigl(\mct_*\ovv,  h^*({\ovg})^{-1}hX\Bigr).
\end{split}
\end{equation}

Let us work out the pullback 
\begin{equation}\label{E:pbOmega}
{\hat\Omega}\stackrel{\rm def}{=} \mct_d^*\Omega.
\end{equation}
For this the crucial observation has been noted in the context of (\ref{E:Tstarvhor}):
$$\hbox{$\mct_*(\ovv)(t)-{\ovv}(t)g_{\ovg}(t)$ is a vertical vector in $T_{\mct(\ovg)(t)}P$,}
$$
where 
$$g_{\ovg}(t)=\tau\bigl(h_{\ovg}(t)\bigr),$$
with $t\mapsto h_{\ovg}(t)$ being the $H$-valued path satisfying
\begin{equation}
h_{\ovg}'(t)h_{\ovg}(t)^{-1}=-C_1\bigl(\ovg'(t)\bigr)
\end{equation}
with initial value $e$.  We have then the parallel transport terms
\begin{equation}\label{E:hstargstar}
\begin{split}
h^*(\ovg) &= h_{\ovg}(t_1)\\
g^*(\ovg) &=\tau\bigl(h^*(\ovg)\bigr).
\end{split}
\end{equation}

Let us also recall that Ad-equivariance of the $L(H)$-valued forms $B_1$ and $C_1^{R,L}$ given in (\ref{propD}) and (\ref{propC}). As  consequence we obtain
\Beq\label{E:flatOmegapb}
\begin{split}
& {\hat\Omega}_{{\ovg}, h}({\ovv}+X) \\
&={\rm Ad}(h^{-1})\left[{\hat\omega}_{(\ovg,h)}(\ovv)+{\rm Ad}\bigl(g^*({\ovg}) \bigr)^{-1} C^R_1|_{\ovg(t_1)}({\ovv}(t_1))- C^L_1|_{\ovg(t_0)}({\ovv}(t_0)) \right. \\
&\left. \hskip 1in + \int_{t_0}^{t_1} {\rm Ad}\bigl(g_{\ovg(t)}^{-1}\bigr)   {B_1}|_{\ovg(t)}({\ovv}(t), {\ovg}'(t)\bigr) \,dt+Xh^{-1}\right],\\
\end{split}
\Eeq
where the term ${\hat\omega}_{(\ovg,h)}(\ovv)$ is given explicitly by
\begin{equation}\label{E:omgamh}
\begin{split}
{\hat\omega}_{(\ovg,h)}(\ovv) &\stackrel{\rm def}{=}
\omega_{\mct({\ovg})}(\mct_*{\ovv}) \\
&=A\bigl(\ovv(t_0)\bigr) +{\rm Ad}\bigl(g^*({\ovg}) \bigr)^{-1}{C^R_0}\bigl(\ovv(t_1)\bigr)- {C^L_0}\bigl(\ovv(t_0)\bigr)\\
&\qquad +\int_{t_0}^{t_1}{\rm Ad}\bigl(g_{\ovg(t)}^{-1}\bigr)   B_0\bigl(\ovv(t), \ovg'(t)\bigr)\,dt.
\end{split}
\end{equation}
Let us recall that, working inside the group $H\rtimes_{\alpha}G$, the action of  $\alpha(g)$, where $g\in G$, on $H$ is just conjugation by $g$, and so ${\rm Ad}(g)$ maps $L(H)$ into itself.

Then the form  $\hat\Omega$, written out in full, is given by
\begin{equation}\label{E:hatOmfull}
\begin{split}
& {\hat\Omega}_{{\ovg}, h}({\ovv}+X) \\
&={\rm Ad}(h^{-1})\left[A\bigl(\ovv(t_0)\bigr) +{\rm Ad}\bigl(g^*({\ovg}) \bigr)^{-1}\bigl({C^R}\bigr)|_{\ovg(t_1)}\bigl(\ovv(t_1)\bigr)\right. \\
&\hskip 1in \left. - {C^L}|_{\ovg(t_0)}\bigl(\ovv(t_0)\bigr)\right. \\
&\left.\hskip  1in+\int_{t_0}^{t_1}{\rm Ad}\bigl(g_{\ovg(t)}^{-1}\bigr) B|_{\ovg(t)}({\ovv}(t), {\ovg}'(t) \,dt +Xh^{-1}\right]
\end{split}
\end{equation}
\subsection{Horizontal lifts by $\hat\omega$ and ${\hat\Omega}$}

Consider a $C^\infty$ map
$$\Gamma:[t_0,t_1]\times [s_0,s_1] \to M:(t,s)\mapsto \Gamma(t,s)=\Gamma_s(t)$$
and an $\ovA$-horizontal lift
$$ \ovG_0:[t_0,t_1]\to P: t\mapsto \ovG_0(t)$$
of the initial path $\Gamma_0$ on $M$. Now let
$$[t_0,t_1]\times [s_0,s_1]\to P:(t,s) \mapsto {\tilde\Gamma}_s(t)$$
be the $\om$-horizontal lift of the path $s\mapsto\Gamma_s$, with initial value being the given path $\ovG_{s_0}$:
$${\tilde\Gamma}_{s_0} =\ovG_{s_0}.$$
Now let
$$\ovG:[t_0,t_1]\times [s_0,s_1]\to P:(t,s)\mapsto\ovG_s(t)$$
be the map for which each path $\ovG_s$ is $\ovA$-horizontal and the initial points $s\mapsto\ovG_s(t_0)$ constitute a path
$$[t_0,t_1]\to P: t\mapsto \ovG_{s_0}(t)$$ 
that is horizontal with respect to the connection for $A$.

Then by  (\ref{E:adotsas})  (applied to the connection form $\hat\om$) it follows that the path ${\tlG}_s$ is obtained from $\ovG_s$ by translation with an element $a_s\in G$:
$$\tlG_s=\ovG_sa_s ,$$
where $s\mapsto a_s$ satisfies the differential equation
\begin{equation}\label{E:flatadotsas}
\begin{split}
 {\dot a}_sa_s^{-1}& = -{\rm Ad}\bigl(g^*(\ovG_s(t_1))^{-1} \bigr){C^R_0}\Bigl(\partial_t{\ovG}_s(t_1)\Bigr) + {C^L_0}\Bigl(\partial_t{\ovG}_s(t_0)\Bigr) \\
 &\hskip 1in -\int_{t_0}^{t_1}{\rm Ad}\bigl(g^*(\ovG_s(t))^{-1} \bigr)  B_0\Bigl( \partial_s{\ovG}_s(t), \partial_t{\ovG}_s(t)\Bigr)\,dt.
  \end{split}
\end{equation}
with initial value $a_{s_0}=e$.   (In  (\ref{E:adotsas}) there is a first term on the right that is absent here because $A=\ovA$ in this context.)

Next, by Proposition \ref{P:Omegahorlift} the path
$$[s_0,s_1]\to \pdbap:s\mapsto (\tlG_s, x_s)$$
is ${\hat\Omega}$-horizontal if and only if the path $s\mapsto x_s\in H$
satisfies the differential equation
\begin{equation}\label{E:flatdothdot2}
\begin{split}
{\dot x}_sx_s^{-1} &=   C^L_1\bigl(\partial_s\tlG_s(t_0)\bigr)- {\rm Ad}\bigl(g^*(\ovG_s(t_1))^{-1} \bigr) C^R_1\bigl(\partial_s\tlG_s(t_1)\bigr)  \\
&  \qquad -\int_{t_0}^{t_1}{\rm Ad}\bigl(g^*(\ovG_s(t))^{-1} \bigr)  B_1\Bigl(\partial_s{ \tlG}_s(t), \partial_t{ \tlG}_s(t)\Bigr)\,dt.
\end{split}
\end{equation}
From the second relation  bwteen $\alpha$ and $\tau$ that have noted  in (\ref{liecross}) we have
\begin{equation}\label{E:alphatau}
{\rm Ad}\bigl(\tau(h)\bigr)X=\alpha\bigl(\tau(h)\bigr)X={\rm Ad}(h)X
\end{equation} 
for all $h\in H$ and $X\in L(H)$. From this we see that in the right side of (\ref{E:flatdothdot2}) we can replace each $g^*$ by an $h^*$, and so
\begin{equation}\label{E:flatdothdot2hstar}
\begin{split}
{\dot x}_sx_s^{-1} &=   C^L_1\bigl(\partial_s\tlG_s(t_0)\bigr)- {\rm Ad}\bigl(h^*(\ovG_s(t_1))^{-1} \bigr) C^R_1\bigl(\partial_s\tlG_s(t_1)\bigr)  \\
&  \qquad -\int_{t_0}^{t_1}{\rm Ad}\bigl(h^*(\ovG_s(t))^{-1} \bigr)  B_1\Bigl(\partial_s{ \tlG}_s(t), \partial_t{ \tlG}_s(t)\Bigr)\,dt.
\end{split}
\end{equation}
 
\subsection{Comparison with variation of parallel transport}  We continue with the framework as above.  We focus now on determining conditions on the forms   $C^{R,L}$  and $B$   that ensure that the path
$$[s_0,s_1]\to \pdbap:s\mapsto \Bigl(\tlG_s, h^*(\tlG_s)\Bigr)$$
is ${\hat\Omega}$-horizontal. To this end let us note first that
\begin{equation}
 h^*(\tlG_s)^{-1}=h_s(t_1),
 \end{equation}
Consider the path
 $$[t_0,t_1]\to H:t\mapsto h_s(t)$$
 that satisfies
 \begin{equation}
 \begin{split}
 h_s'(t) + C_1\bigl({\tilde\Gamma}_s'(t)\bigr)h_s(t) &=0 \\
 h_s(t_0) &= e,
 \end{split}
 \end{equation}
 and let
 \begin{equation}\label{E:hxhs}
 y_s(t)\stackrel{\rm def}{=} h_s(t)^{-1}.
 \end{equation}
 Then
 \begin{equation}\label{E:xdots}
 \begin{split}
 y_s(t)^{-1} {  y}'_s(t)&=  h_s(t)\bigl(-h_s(t)^{-1}{   h}'_s(t)h_s(t)^{-1}\bigr)\\
 &=
 -{  h}'_s(t)h_s(t)^{-1}\\
 &= C_1\bigl({\tilde\Gamma}_s'(t)\bigr).
 \end{split}
 \end{equation}
 Then, as we show below in (\ref{E:Dst3}), 
  \begin{equation}\label{E:hsexpress}
 \begin{split}
  & {\dot y}_s(t_1) y_s(t)^{-1} - {\dot y}_s(t_0) y_s(t_0)^{-1}\\
 &=- \int_{t_0}^{t_1} y_s(u)\Bigl(d{{\tilde C}_1}+\frac{1}{2}[{{\tilde C}_1},{{\tilde C}_1}]\Bigr)(\partial_t,\partial_s)y_s(u)^{-1} du \\
 &\qquad\qquad + y_s(t_1){  C}_1\bigl(\partial_s\tlG_s(t_1)\bigr)y_s(t_1)^{-1} -y_s(t_0) {  C}_1\bigl(\partial_s\tlG_s(t_0)\bigr)y_s(t_0)^{-1},
  \end{split}
 \end{equation}
 where
 \begin{equation}\label{E:CtildGam}
 {\tilde C}_1={\tilde\Gamma}^*C_1.
 \end{equation}
 Since $y_s(t_0)$ is held fixed at $e$ and   
 \begin{equation}\label{E:hsstarxs}
 y_s(t_1)=h^*(\tlG_s)^{-1},
 \end{equation}
 we have then, on using (\ref{E:xdots}),
  \begin{equation}\label{E:hsexpress2}
 \begin{split}
  &  \partial_sh^*(\tlG_s)h^*(\tlG_s)^{-1} \\
 &= h^*(\tlG_s)^{-1}{  C}_1\bigl(\partial_s\tlG_s(t_1)\bigr) h^*(\tlG_s) - {  C}_1\bigl(\partial_s\tlG_s(t_0)\bigr)\\
 &\qquad\qquad +   \int_{t_0}^{t_1} h_s(u)^{-1}\Bigl(d{{\tilde C}_1}+\frac{1}{2}[{{\tilde C}_1},{{\tilde C}_1}]\Bigr)(\partial_t,\partial_s)h_s(u)\, du.  \end{split}
 \end{equation}
Comparing with the equation for $\hat\Omega$-parallel transport (\ref{E:flatdothdot2}
) we see that the two equations agree if
\begin{equation}\label{E:C0C1}
\begin{split}
 C_1^L= C_1^R &=-{  C}_1\\
B_1 &=-\Bigl(d{{  C}_1}+\frac{1}{2}[{{  C}_1},{{  C}_1}]\Bigr).
\end{split}
\end{equation}

We have thus completed the proof of  Proposition \ref{P:reduction}.

 \subsection{Variation of  differential equations}\label{ss:vardiff}
 In this subsection we work through the details of the computation that leads to the equation (\ref{E:hsexpress2}) which was central to the proof of Proposition \ref{P:reduction}.

Consider the differential equation
\begin{equation}
b(t)^{-1} b'(t) = C\bigl({\tilde\gamma}'(t)\bigr)  
 \end{equation}
 for $t\in [0,1]$.
We shall determine  how fast the terminal point $b(1)$ changes when we change the path ${\tilde\gamma}$. 

Our  strategy is to consider the family of differential equations
\begin{equation}\label{E:bprims}
b_s(t)^{-1} b_s'(t) =   C\bigl({\tilde\Gamma}_s'(t)\bigr)\end{equation}
where $s\in [0,1]$ and $t\in [0,1]$, and the prime is derivative with respect to $t$. Here
\begin{equation} {\tilde\Gamma}: [0,1]\times [0,1]\to P: (t,s)\mapsto {\tilde\Gamma}_s(t)
\end{equation}
is a smooth map. We think of $s$ as a variational parameter and our goal is to compute how fast $b_s(1)$ changes with $s$.

Let us denote the right hand  side by $E_s(t)\in L(G)$:
\begin{equation}\label{E:Est} E_s(t) = C\bigl({\tilde\Gamma}_s'(t)\bigr).
\end{equation}
Thus our differential equation is
\begin{equation}b_s(t)^{-1} b'_s(t)=E_s(t).
\end{equation}
Now let
\begin{equation}\label{E:Dst} D_s(t) ={\dot b}_s(t)b_s(t)^{-1},
\end{equation}
where 
\begin{equation}\label{E:dotbs}
{\dot b}_s(t)=\partial_sb_s(t)
\end{equation}
is the derivative which contains the information we are ultimately seeking.  Our goal is to compute $D_s(t)$.

As always, we will  the $s$-derivative by a dot over the letter:
\begin{equation}\label{E:dotxs}
 {\dot x}_s(t)={\partial_s}x_s(t).
\end{equation}

Our strategy is to compute $D_s'(t)=\partial_tD_s(t)$ and then obtain $D_s(t)$ by integrating:
$$D_s(t)= \int_0^t D'_s(u)\,du +D_s(0).$$
So now let us compute the derivative $D'_s(t)$. From (\ref{E:Dst} ) we have
\begin{equation}\label{E:Dprimst}
 D'_s(t) =-{\dot b}_s(t) b_s(t)^{-1}b'_s(t)b_s(t)^{-1}+\bigl(\partial_s\partial_tb_s(t)\bigr)b_s(t)^{-1}
\end{equation}

 Now we are going to work out $\partial_sE_s(t)$, but first let us recall what $E_s(t)$ is:
\begin{equation}\label{E:Est2}
 E_s(t) =b_s(t)^{-1}b'_s(t).
\end{equation}
It is  important that we have  $b_s(t)^{-1}$ on the left for $E_s(t)$ and on the right for $D_s(t)$. Returning to the calculation, we have:
\begin{equation}\label{E:Edelst}
  {\dot E}_s(t) = b_s(t)^{-1} \partial_s\partial_tb_s(t) -b_s(t)^{-1}{\dot b}_s(t)b_s(t)^{-1}b'_s(t).
 \end{equation}
Comparing with $D'_s(t)$  we see that it is useful to conjugate ${\dot E}_s(t)$ by $b_s(t)$:
\begin{equation}\label{E:bstinvEprimst}
 b_s(t){\dot E}_s(t)b_s(t)^{-1} = \bigl(\partial^2_{st}b_s(t)\bigr) b_s(t)^{-1}-{\dot b}_s(t)b_s(t)^{-1}b'_s(t)b_s(t)^{-1},\end{equation}
which is exactly $D'_s(t)$! Thus:
\begin{equation}\label{E:Dprimst2}
 D'_s(t) =b_s(t){\dot E}_s(t)b_s(t)^{-1}.\end{equation}
Integrating, we obtain
\begin{equation}\label{E:Dst2} D_s(t) =D_s(0)+\int_0^t b_s(u){\dot E}_s(u) b_s(u)^{-1}\,du.
\end{equation}
This is in itself a nice formula  for $D_s(t)={\dot b}_s(t)b_s(t)^{-1}$,
the rate of change of $b_s(t)$ when $s$ is varied. 

Let us  formally summarize what we have proved so far as a self-contained result.

\begin{prop}\label{P:nonabstokes}
Let $H$ be a Lie group and 
$$E:[t_0,t_1]\times [s_0,s_1]\to L(H):(t,s)\mapsto E_s(t)$$
 a $C^1$ mapping. Suppose
 $$[t_0,t_1]\times [s_0,s_1]\to H:(t,s)\mapsto b_s(t)$$
 be a function that satisfies
 \begin{equation}\label{E:bprims2}
 b_s(t)^{-1}b_s'(t)  =E_s(t)\ \end{equation}
 for all $(t,s)\in [t_0,t_1]\times [s_0,s_1]$. 
 Then 
 \begin{equation}\label{E:brpmt}
 {\dot b}_s(t)b_s(t)^{-1} ={\dot b}_s(t_0)b_s(t_0)^{-1}+\int_{t_0}^t{\rm Ad}\bigl(b_s(u)\bigr){\dot E}_s(u)\,du
 \end{equation}
 for all $(t,s)\in [t_0,t_1]\times [s_0,s_1]$, with a dot over a letter denoting the derivative with respect to $s$.
\end{prop}
Results of this type are sometimes called `non-abelian Stokes formulas.'

Now let us return to the geometric context, with notation as before. The specific form   (\ref{E:Est})  for $E_s(t)$:
\begin{equation}\label{E:bEst2}
 E_s(t)=  C\Bigl(\partial_t{\tilde\Gamma}_s(t)\Bigr).\end{equation}
We can write this also as
\begin{equation}\label{E:bEst3}
  E_s(t) = {{\tilde C}}(\partial_t),\end{equation}
where ${{\tilde C}}$ is the pull back

\begin{equation}\label{E:CGam}
 {{\tilde C}}={\tilde\Gamma}^*C,\end{equation}
which is a 1-form on $[0,1]\times[0,1]$. Let us now use the formula for the exterior differential of a 1-form:
\begin{equation}\label{E:dCGam}
  d{{\tilde C}}(v,w) =v[{{\tilde C}}(w)] -w[{{\tilde C}}(v)] -{{\tilde C}}([v,w]),\end{equation}
for any smooth vector fields $v$ and $w$. Then
\begin{equation}\label{E:dCGam2}
  d{{\tilde C}}(\partial_t,\partial_s) = \partial_t\Bigl({{\tilde C}}(\partial_s)\Bigr)-\partial_s\Bigl({{\tilde C}}(\partial_t)\Bigr)-{{\tilde C}}\Bigl([\partial_t,\partial_s]\Bigr).\end{equation}
The Lie bracket  of the coordinate vector fields $\partial_t$ and $\partial_s$ appearing on the  right is $0$. So we have
\begin{equation}\label{E:partialsEst}
 {\dot E}_s(t) =   \partial_s\Bigl[{{\tilde C}}(\partial_t)\Bigr] = \partial_t\Bigl[{{\tilde C}}(\partial_s)\Bigr]-d{{\tilde C}}(\partial_t, \partial_s).\end{equation}
 To keep the notation simple let us write
\begin{equation}\label{E:Fst}
  F_s(t) ={{\tilde C}}_{(t,s)}(\partial_s).\end{equation}
Then
\begin{equation}\label{E:Edotst}
  {\dot E}_s(t) = - d{{\tilde C}}(\partial_t, \partial_s) + \partial_t F_s(t).\end{equation}
Looking back at $D'_s(t)$ as given in (\ref{E:Dprimst2}) we compute
\begin{equation}\label{E:Dprimst3}
   D_s'(t) =b_s(t)\Bigl(- d{{\tilde C}}(\partial_t, \partial_s)+ F_s'(t) \Bigr)b_s(t)^{-1}.\end{equation}
We focus for now on the second term  and compute:
\begin{equation}\label{E:bstFst}
\begin{split}
  &
  b_s(t)F'_s(t)b_s(t)^{-1} \\
  &=\partial_t\Bigl(b_s(t)F_s(t)b_s(t)^{-1}\Bigr) - b'_s(t) F_s(t)b_s(t)^{-1}-b_s(t)F_s(t)\partial_t\bigl(b_s(t)^{-1}\bigr)\\
& = \partial_t\Bigl(b_s(t)F_s(t)b_s(t)^{-1}\Bigr) - b_s(t) E_s(t)F_s(t)b_s(t)^{-1}    + b_s(t) F_s(t)E_s(t)b_s(t)^{-1}\\
&= \partial_t\Bigl(b_s(t)F_s(t)b_s(t)^{-1}\Bigr) -b_s(t)[E_s(t), F_s(t)] b_s(t)^{-1}
\end{split}
\end{equation}
Let us analyze the Lie bracket term
\begin{equation}\label{E:LieEsFst}
   [E_s(t), F_s(t) ] = [{{\tilde C}}(\partial_t), {{\tilde C}}(\partial_s)]\end{equation}
The 2-form $[{\tilde C},{\tilde C}]$ is defined by
\begin{equation}\label{E:brackCC}
 [{\tilde C},{\tilde C}](v,w) =[{\tilde C}(v), {\tilde C}(w)] -[{\tilde C}(w), {\tilde C}(v)]=2 [{\tilde C}(v), {\tilde C}(w)].
 \end{equation}
 There is also the related notation
 \begin{equation}\label{E:CwedgeC}
({\tilde C}\wedge {\tilde C})(v,w) =[{\tilde C}(v), {\tilde C}(w)],
 \end{equation}
 which is directly meaningful if ${\tilde C}$ takes values in a matrix Lie algebra.

Hence
\begin{equation}\label{E:LieEsFst2}
   [E_s(t), F_s(t) ] =  \frac{1}{2}[{{\tilde C}}, {{\tilde C}}](\partial_t,\partial_s).
\end{equation}
Now glancing back a few steps at (\ref{E:Dprimst3}) we see that
\begin{equation}\label{E:Dprimst4}
\begin{split}
   D_s'(t)& =- b_s(t)d{{\tilde C}}(\partial_t, \partial_s) b_s(t)^{-1} +\partial_t\Bigl(b_s(t)F_s(t)b_s(t)^{-1}\Bigr) \\
   &\hskip 1in -b_s(t)\frac{1}{2}[{{\tilde C}},{{\tilde C}}](\partial_t, \partial_s)b_s(t)^{-1}.
   \end{split}
   \end{equation}
 Integrating,  and recalling from (\ref{E:Dst}) that $D_s(t)$ is ${\dot b}_s(t)b_s(t)^{-1}$, we have
\begin{equation}\label{E:Dst3}
\begin{split}
  {\dot b}_s(t)b_s(t)^{-1}- {\dot b}_s(0)b_s(0)^{-1} 
 &=- \int_0^t b_s(u)\Bigl(d{{\tilde C}}+\frac{1}{2}[{{\tilde C}},{{\tilde C}}]\Bigr)(\partial_t,\partial_s)b_s(u)^{-1} du \\
 &\qquad + b_s(t)F_s(t)b_s(t)^{-1}-  b_s(0)F_s(0)b_s(0)^{-1},
  \end{split}
\end{equation}
wherein, as before, $F_s(t) ={{\tilde C}}_{(t,s)}(\partial_s)$.

\subsection{The holonomy bundle}\label{ss:holb} 
For a principal $G$-bundle  $\pi:P\to M$ equipped with a connectiom $A$  we denote by $P_A(u)$ the set of all terminal points of $A$-horizontal paths that initiate at any given point $u\in P$.  The {\em holonomy group} $H_A(u)$ consists of all $g\in G$ for which $ug\in P_A(u)$.   If $\ovg_u^p$ is an $A$-horizontal path on $P$ initiating at $u$ and terminating at $p\in P_A(u)$ then ${\ovg}_u^pg$ is also $A$-horizontal, initiating at the point $ug$ and terminating at $pg$; if $g\in H_A(u)$ then we can choose an $A$-horizontal path $\ovg_u^{ug}$ from $u$ to $ug$, and the composite $(\ovg_u^pg)\circ \ovg_u^{ug}$ is an $A$-horizontal path from $u$ to the point $pg$. Thus $P_A(u)$ is mapped into itself by the right action of the holonomy group $H_A(u)\subset G$.  In this way the structure 
\begin{equation}
\pi:P_A(u)\to M:p\mapsto \pi(p)
\end{equation}
is a principal $H_A(u)$-bundle over $M$. The connection $A$ reduces to a connection on this bundle.  A celebrated result of Ambrose and Singer  \cite{AmSi1953} relates the Lie algebra of the holonomy group to the Lie subalgebra of $L(G)$ spanned by elements $F^A(v,w)$, where $F^A$ is the curvature of $A$ and $v$ and $w$ run over all vectors in $T_pP$ with $p$ running over the holonomy bundle $P_A(u)$. (Since composition of paths is crucial in these discussions, such as even to see that $H_A(u)$ is a subgroup, the definition of the holonomy bundle should involve a family of paths that is closed under composition; in fact we may use just the type of paths we have been using, $C^\infty$ and constant near the initial and final times.)

\section{Differential calculus on path spaces}\label{s:diffcalc}

We have avoided putting a manifold structure on the spaces of paths with which we have worked.  Such a structure  is not logically needed for any of our constructions and is useful  only as an idea.  It is standard practice in the theory of stochastic processes (which is concerned with integration on path spaces)  to work primarily with notions of differentiation and integration defined in the specific context of path or function spaces rather than on any abstract infinite dimensional manifold. Although an abstract theory of such integration was constructed  (Kuo \cite{Kuo1971}) it has been found to be more useful to define geometric, differential and measure theoretic notions directly on path spaces. Let us then summarize here the differential notions we need for our work.

Consider a set $X$ whose points are paths on a given manifold $M$. In this context we require that the paths be $C^\infty$, and there might be additional restrictions placed.

By a  {\em tangent vector} $v$ to $X$ at a point  in $X$ given by a path $\gamma:[t_0,t_1] \to M$ we mean a $C^\infty$ vector field $v:[t_0,t_1]\to TM$ along $\gamma$ that is constant near $t_0$ and near $t_1$.  For example, there is the special vector $\gamma'\in T_{\gamma}X$ which is just the tangent vector field along $\gamma$ (the tangent vector field along $\gamma$ is zero near the initial and final times). We denote the set of all   vectors tangent to $X$ at $\gamma$ by $T_{\gamma}X$ and call this the {\em tangent space} to $X$ at $\gamma$.   This is clearly a vector space under pointwise addition and scaling.

If $v$ is a $C^\infty$ vector field on an open subset of $M$ and $\gamma\in {\mathcal P}(M)$ lies entirely in $U$ then we obtain a vector field $v_{\gamma}$ along $\gamma$ given by
$$v_{\gamma}(t)=v\bigl(\gamma(t)\bigr)\quad\hbox{for all $t\in [t_0, t_1]$.}$$
Then $v_{\gamma}$ is $C^\infty$ and constant near $t_0$ and $t_1$, and hence is a vector in the tangent space $T_{\gamma}{\mathcal P}(M)$. 

A {\em  $k$-form} $\Theta$ on $X$ is an assignment to each $\gamma\in X$ an alternating multilinear mapping  
$$(T_{\gamma}X)^k\to\mbr: (v_1,\ldots, v_k)\mapsto \Theta_{\gamma}(v_1,\ldots, v_k).$$
A typical example of interest  is a $k$-form $I(\theta)$ that arises from a $k$-form $\theta$ on $M$ by
the specification:
\begin{equation}\label{E:Ith}
I(\theta)_{\gamma}(v_1,\ldots, v_k)=\int_{t_0}^{t_1} \theta_{\gamma(t)}\bigl( v_1(t),\ldots, v_k(t)\bigr)\,dt.
\end{equation}
Many forms  on $X$ of interest to us have some additional features: for example, they are invariant under a class of reparametrizations of the paths. Moreover, many of the forms we use vanish when contracted on the tangent vector field. As an example consider, with $I(\theta)$ as above,   the $(k-1)$-form on $X$ given by: 
\begin{equation}\label{E:CIth}
(i_{\gamma'}I(\theta)_{\gamma})(v_1,\ldots, v_{k-1})=\int_{t_0}^{t_1} \theta_{\gamma(t)}\bigl(\gamma'(t), v_1(t),\ldots, v_{k-1}(t)\bigr)\,dt.
\end{equation}
This form vanishes when one of the vectors $v_j$ happens to be $\gamma'$. The form $i_{\gamma'}I(\theta)_{\gamma}$ is the {\em Chen integral} 
\begin{equation}\label{E:ChenIth}
\int_{\gamma}\theta\stackrel{\rm def}{=} i_{\gamma'}I(\theta)_{\gamma}.
\end{equation}

Intuitively we think of $X$ as a bundle over a quotient space $[X]$ after quotienting by a  group  of reparametrizations. Of interest then are forms on $X$ that vanish along the orbital directions and are invariant under translations  (reparametrizations) by the action of the structure group; thus these correspond to forms on $[X]$ pulled back up to  the space $X$.

Now let
$$\Gamma:[t_0,t_1]\times [s_0,s_1]\to M: (t,s)\to \Gamma_s(t) $$
be a $C^\infty$ map  which is stationary near the boundary in the following sense:   there is an $\epsilon>0$ such that for each fixed $s$ the point  $\Gamma_s(t)$ is the same when $t$ is at distance $<\epsilon$ from $\{t_0, t_1\}$, and for each fixed $t$ the point $\Gamma_s(t)$ is the same when $s$ is at distance $<\epsilon$ from $\{s_0, s_1\}$. Thus  each $\Gamma_s$ is in ${\mathcal P}(M)$ as defined in (\ref{E:parmpath}).  Then there is for each $s\in [s_0,s_1]$ the tangent vector ${\dot\Gamma}_s\in T_{\Gamma_s}{\mathcal P}(M)$ given by
\begin{equation}\label{E:Gdotst}
{\dot\Gamma}_s(t)=\partial_s\Gamma_s(t) \qquad\hbox{for all $t\in [t_0,t_1]$.}
\end{equation}
Other differential geometric notions such as bundles and connections over spaces of paths can be defined by natural extension of the usual definitions on finite dimensional spaces. It is interesting to note that for these considerations (aside from local triviality) we do not need a topology on the space of paths.

 \section{Concluding Remarks}\label{s:concr}

  In this paper we have studied the differential geometry of bundles in which the points of the bundle space are pairs of the form $(\ovg,h)$, where $\ovg$ is a path on a given principal bundle that is horizontal with respect to a given connection form and $h$ is a `decorating' element lying in a second structure group. We analyzed parallel transport in such path space bundles. These investigations are motivated by physical theories strings and particles interacting via traditional gauge fields as well as `higher' gauge fields operating over the space of paths (strings). In forthcoming work we will explore in greater detail local trivializations of the decorated bundle as well as the relationship with the theory of gerbes. 
   
 {\bf{ Acknowledgments.} }   Sengupta   acknowledges  (i) research support from   NSA grant H98230-13-1-0210,  and (ii) the SN Bose National Centre for its hospitality during visits when part of this work was done.  Chatterjee acknowledges support through a fellowship from the Jacques Hadamard Mathematical Foundation.

\end{document}

%% file: gammaba2r.pstex_t
\begin{picture}(0,0)%
\includegraphics{gammaba2r.pstex}%
\end{picture}%
\setlength{\unitlength}{4144sp}%
\begingroup\makeatletter\ifx\SetFigFont\undefined%
\gdef\SetFigFont#1#2#3#4#5{%
  \reset@font\fontsize{#1}{#2pt}%
  \fontfamily{#3}\fontseries{#4}\fontshape{#5}%
  \selectfont}%
\fi\endgroup%
\begin{picture}(5562,5029)(391,-7678)
\put(2566,-3211){\makebox(0,0)[lb]{\smash{{\SetFigFont{14}{16.8}{\rmdefault}{\mddefault}{\updefault}{\color[rgb]{0,0,0}${\bar \gamma}\in {\mathcal P}_{\bar A}P$}%
}}}}
\put(3151,-6766){\makebox(0,0)[lb]{\smash{{\SetFigFont{14}{16.8}{\rmdefault}{\mddefault}{\updefault}{\color[rgb]{0,0,0}$\gamma\in {\mathcal P}M$}%
}}}}
\put(4186,-7486){\makebox(0,0)[lb]{\smash{{\SetFigFont{14}{16.8}{\rmdefault}{\mddefault}{\updefault}{\color[rgb]{0,0,0}$M$}%
}}}}
\put(406,-4426){\makebox(0,0)[lb]{\smash{{\SetFigFont{14}{16.8}{\rmdefault}{\mddefault}{\updefault}{\color[rgb]{0,0,0}$P$}%
}}}}
\put(3736,-4381){\makebox(0,0)[lb]{\smash{{\SetFigFont{14}{16.8}{\rmdefault}{\mddefault}{\updefault}{\color[rgb]{0,0,0}$\pi$}%
}}}}
\put(1801,-4156){\makebox(0,0)[lb]{\smash{{\SetFigFont{14}{16.8}{\rmdefault}{\mddefault}{\updefault}{\color[rgb]{0,0,0}${\bar \gamma}(t_0)$}%
}}}}
\put(4816,-3526){\makebox(0,0)[lb]{\smash{{\SetFigFont{14}{16.8}{\rmdefault}{\mddefault}{\updefault}{\color[rgb]{0,0,0}${\bar \gamma}(t_1)$}%
}}}}
\end{picture}%

%% file: pardeco.pstex_t
\begin{picture}(0,0)%
\includegraphics{pardeco.pstex}%
\end{picture}%
\setlength{\unitlength}{4144sp}%
\begingroup\makeatletter\ifx\SetFigFont\undefined%
\gdef\SetFigFont#1#2#3#4#5{%
  \reset@font\fontsize{#1}{#2pt}%
  \fontfamily{#3}\fontseries{#4}\fontshape{#5}%
  \selectfont}%
\fi\endgroup%
\begin{picture}(5109,5524)(844,-7678)
\put(3151,-6766){\makebox(0,0)[lb]{\smash{{\SetFigFont{14}{16.8}{\rmdefault}{\mddefault}{\updefault}{\color[rgb]{0,0,0}$\gamma\in {\mathcal P}M$}%
}}}}
\put(4186,-7486){\makebox(0,0)[lb]{\smash{{\SetFigFont{14}{16.8}{\rmdefault}{\mddefault}{\updefault}{\color[rgb]{0,0,0}$M$}%
}}}}
\put(1801,-4156){\makebox(0,0)[lb]{\smash{{\SetFigFont{14}{16.8}{\rmdefault}{\mddefault}{\updefault}{\color[rgb]{0,0,0}${\bar \gamma}(t_0)$}%
}}}}
\put(5131,-4876){\makebox(0,0)[lb]{\smash{{\SetFigFont{14}{16.8}{\rmdefault}{\mddefault}{\updefault}{\color[rgb]{0,0,0}$\pi^{-1}({\gamma}(t_1))\subset P$}%
}}}}
\put(4141,-4111){\makebox(0,0)[lb]{\smash{{\SetFigFont{14}{16.8}{\rmdefault}{\mddefault}{\updefault}{\color[rgb]{0,0,0}$\pi_{\bar A}$}%
}}}}
\put(3016,-3976){\makebox(0,0)[lb]{\smash{{\SetFigFont{14}{16.8}{\rmdefault}{\mddefault}{\updefault}{\color[rgb]{0,0,0}$\pi^{dec}_{\bar A}$}%
}}}}
\put(2071,-4921){\makebox(0,0)[lb]{\smash{{\SetFigFont{14}{16.8}{\rmdefault}{\mddefault}{\updefault}{\color[rgb]{0,0,0}$\pi^{-1}({\gamma}(t_0))\subset P$}%
}}}}
\put(4861,-3481){\makebox(0,0)[lb]{\smash{{\SetFigFont{14}{16.8}{\rmdefault}{\mddefault}{\updefault}{\color[rgb]{0,0,0}${\bar \gamma}(t_1)$}%
}}}}
\put(4996,-2896){\makebox(0,0)[lb]{\smash{{\SetFigFont{14}{16.8}{\rmdefault}{\mddefault}{\updefault}{\color[rgb]{0,0,0}${\bar \gamma}(t_1)\tau(h)$}%
}}}}
\put(2701,-2716){\makebox(0,0)[lb]{\smash{{\SetFigFont{14}{16.8}{\rmdefault}{\mddefault}{\updefault}{\color[rgb]{0,0,0}$({\bar {\gamma}}, h)\in {\mathcal P}_{\bar A}P\times H$}%
}}}}
\put(3286,-3301){\makebox(0,0)[lb]{\smash{{\SetFigFont{14}{16.8}{\rmdefault}{\mddefault}{\updefault}{\color[rgb]{0,0,0}${\bar \gamma}\in {\mathcal P}_{\bar A}P$}%
}}}}
\end{picture}%

%% file: higherpar.pstex_t
\begin{picture}(0,0)%
\includegraphics{higherpar.pstex}%
\end{picture}%
\setlength{\unitlength}{4144sp}%
\begingroup\makeatletter\ifx\SetFigFont\undefined%
\gdef\SetFigFont#1#2#3#4#5{%
  \reset@font\fontsize{#1}{#2pt}%
  \fontfamily{#3}\fontseries{#4}\fontshape{#5}%
  \selectfont}%
\fi\endgroup%
\begin{picture}(5919,5919)(349,-7498)
\put(1531,-6091){\makebox(0,0)[lb]{\smash{{\SetFigFont{14}{16.8}{\rmdefault}{\mddefault}{\updefault}{\color[rgb]{0,0,0}$\Gamma_0\in {\mathcal P}M$}%
}}}}
\put(4951,-6226){\makebox(0,0)[lb]{\smash{{\SetFigFont{14}{16.8}{\rmdefault}{\mddefault}{\updefault}{\color[rgb]{0,0,0}$\Gamma_1\in {\mathcal P}M$}%
}}}}
\put(1441,-2446){\makebox(0,0)[lb]{\smash{{\SetFigFont{14}{16.8}{\rmdefault}{\mddefault}{\updefault}{\color[rgb]{0,0,0}$({\bar {\Gamma}}_0, h_0)\in {\mathcal P}_{\bar A}P\times H$}%
}}}}
\put(586,-7261){\makebox(0,0)[lb]{\smash{{\SetFigFont{14}{16.8}{\rmdefault}{\mddefault}{\updefault}{\color[rgb]{0,0,0}$M$}%
}}}}
\put(4321,-2671){\makebox(0,0)[lb]{\smash{{\SetFigFont{14}{16.8}{\rmdefault}{\mddefault}{\updefault}{\color[rgb]{0,0,0}$({\bar {\Gamma}_1}, h_1)\in {{\mathcal P}_{\bar A}P}\times H$}%
}}}}
\put(3241,-3166){\makebox(0,0)[lb]{\smash{{\SetFigFont{14}{16.8}{\rmdefault}{\mddefault}{\updefault}{\color[rgb]{0,0,0}$({\bar \Gamma}, h, k)$}%
}}}}
\put(2836,-5776){\makebox(0,0)[lb]{\smash{{\SetFigFont{14}{16.8}{\rmdefault}{\mddefault}{\updefault}{\color[rgb]{0,0,0}$\Gamma:[s_0,s_1]\to {\mathcal P}M$}%
}}}}
\end{picture}%